\documentclass[a4paper]{amsart}

\usepackage[utf8]{inputenc}
\usepackage[T1]{fontenc}

\usepackage{lmodern}
\usepackage[osf]{mathpazo}
\usepackage{eucal}

\usepackage{amssymb}
\usepackage{caption}
\usepackage[dvipsnames]{xcolor}
\usepackage{graphicx}
\usepackage{ifthen}
\usepackage{mathtools}
\usepackage[]{microtype}
\UseMicrotypeSet[protrusion]{basicmath}
\usepackage{nicefrac}
\usepackage{stmaryrd}
\SetSymbolFont{stmry}{bold}{U}{stmry}{m}{n}
\usepackage{svg}
\usepackage{tabularx}
\usepackage{tikz}
\usepackage{tikz-cd}
\usepackage{verbatim}
\usepackage{xspace}

\usepackage{enumitem}
\setlist[enumerate,1]{label=(\roman*), font=\normalfont}

\usepackage{hyperref}
\makeatletter
\@ifpackageloaded{hyperref}{

\colorlet{mycitecolor}{violet}
\colorlet{mylinkcolor}{Green}
\colorlet{myurlcolor}{Aquamarine}

\definecolor{dark-red}{rgb}{0.5,0.15,0.15}
\definecolor{dark-blue}{rgb}{0.15,0.15,0.5}
\definecolor{medium-blue}{rgb}{0,0,0.5}
\hypersetup{
    colorlinks, linkcolor={dark-red},
    citecolor={dark-blue}, urlcolor={medium-blue}
}

}
{
\newcommand{\texorpdfstring}[2]{#1} 
\newcommand{\phantomsection}{} 
}

\def\namedlabel#1#2{\begingroup
  #2%
  \def\@currentlabel{#2}%
  \phantomsection\label{#1}\endgroup
}
\makeatother

\usepackage[initials, alphabetic]{amsrefs}
\usepackage{cleveref}

\newtheorem{theorem}{Theorem}[section]
\newtheorem{cor}[theorem]{Corollary}
\newtheorem{lemma}[theorem]{Lemma}
\newtheorem{prop}[theorem]{Proposition}

\newtheorem{fact}[theorem]{Fact}

\newtheorem*{theorem*}{Theorem}

\theoremstyle{definition}
\newtheorem{defn}[theorem]{Definition}
\newtheorem{example}[theorem]{Example}
\newtheorem{question}[theorem]{Question}

\theoremstyle{remark}
\newtheorem{remark}[theorem]{Remark}

\newtheorem*{claim*}{Claim}

\numberwithin{equation}{section}

\newcommand{\myenumlabel}[1]{\textnormal{(\roman{#1})}}

\newcounter{cycprfcnt}
\newcounter{cycprffirst}
\newcommand{\cycprfpreamble}%
{%
  \setcounter{cycprfcnt}{1}
  \setcounter{cycprffirst}{0}
  \setlength{\itemindent}{0.5\leftmargin}%
  \setlength{\leftmargin}{0pt}%
  \newcommand{\cpcurr}{\myenumlabel{cycprfcnt}}%
  \newcommand{\cpnext}{\addtocounter{cycprfcnt}{1}\cpcurr}%
  \newcommand{\cpnum}[1]{\setcounter{cycprfcnt}{##1}\cpcurr}%
  \newcommand{\cpfirst}{\cpnum{1}}%
  \newcommand{\impnext}{\cpcurr{} $\Rightarrow$ \cpnext.}%
  \newcommand{\impfirst}{\cpcurr{} $\Rightarrow$ \cpfirst.}%
  \def\makelabel##1{\ifnum\value{cycprffirst}=0\hspace{-0.7\itemindent}\setcounter{cycprffirst}{1}\fi##1}%
}%

\newenvironment{cycprf}%
{\begin{list}{\impnext}%
  {\cycprfpreamble}}%
{\qedhere\end{list}}%

\newenvironment{cycprf*}%
{\begin{list}{\impnext}%
  {\cycprfpreamble}}%
{\end{list}}

\DeclarePairedDelimiter{\set}{\{}{\}}
\DeclarePairedDelimiterX\setnew[2]{\{}{\}}{#1 \nonscript\;\delimsize \vert \nonscript\; #2}

\newcommand{\acts}{\curvearrowright}
\newcommand{\aut}{\mathrm{Aut}}

\newcommand{\N}{\mathbb N}

\newcommand{\restr}[1]{\upharpoonright #1}

\newcommand{\df}{\emph}

\newcommand{\bA}{{\mathbf A}}
\newcommand{\bB}{{\mathbf B}}
\newcommand{\bC}{{\mathbf C}}
\newcommand{\bD}{{\mathbf D}}

\newcommand{\bM}{{\mathbf M}}
\newcommand{\bN}{{\mathbf N}}

\newcommand{\bQ}{{\mathbf Q}}

\newcommand{\bS}{{\mathbf S}}
\newcommand{\bT}{{\mathbf T}}
\newcommand{\bU}{{\mathbf U}}

\newcommand{\bX}{{\mathbf X}}
\newcommand{\bY}{{\mathbf Y}}

\newcommand{\cF}{\mathcal{F}}

\newcommand{\cK}{\mathcal{K}}
\newcommand{\cL}{\mathcal{L}}
\newcommand{\cM}{\mathcal{M}}
\newcommand{\cN}{\mathcal{N}}

\newcommand{\cS}{\mathcal{S}}

\newcommand{\cU}{\mathcal{U}}

\newcommand{\eps}{\epsilon}

\newcommand{\sub}{\subseteq}
\newcommand{\sminus}{\setminus}

\DeclarePairedDelimiter\gen{\langle}{\rangle}

\DeclarePairedDelimiter\abs{\lvert}{\rvert}

\newcommand{\cl}[2][]{\overline{#2}^{#1}}

\newcommand{\actson}{\curvearrowright}

\providecommand{\dotminus}{} 
\renewcommand{\dotminus}{\mathbin{\ooalign{\hss\raise1ex\hbox{.}\hss\cr
      \mathsurround=0pt$-$}}}

\newcommand{\findep}[1][]{%
  \mathrel{
    \mathop{
      \vcenter{
        \hbox{\oalign{\noalign{\kern-.3ex}\hfil$\vert$\hfil\cr
              \noalign{\kern-.7ex}
              $\smile$\cr\noalign{\kern-.3ex}}}
      }
    }\displaylimits_{#1}
  }
}  

\DeclareMathOperator{\id}{id}
\DeclareMathOperator{\tp}{tp}
\DeclareMathOperator{\tS}{S}

\DeclareMathOperator{\dcl}{dcl}

\DeclareMathOperator{\Aut}{Aut}
\DeclareMathOperator{\Homeo}{Homeo}

\DeclareMathOperator{\Sym}{Sym}
\DeclareMathOperator{\Th}{Th}
\DeclareMathOperator{\Age}{Age}

\DeclareMathOperator{\Succ}{Succ}
\DeclareMathOperator{\Exp}{Exp}

\DeclareMathOperator{\UCB}{UCB}

\renewcommand{\Im}{\mathrm{Im\,}}

\newcommand{\Fraisse}{Fra{\"\i}ss{\'e}\xspace}
\newcommand{\Wazewski}{Wa\.{z}ewski\xspace}

\renewcommand{\And}{\text{ and }}

\newcommand{\umf}{\cM}

\newcommand{\tort}{\overset{r}{\to}}
\newcommand{\divby}{\mkern+2.5mu\reflectbox{/}\mkern+1mu }

\newcommand{\waz}{W}

\newcommand{\reg}{\mathrm{Reg}}
\newcommand{\theends}{\mathrm{End}}
\newcommand{\comps}[1]{\hat{#1}}
\newcommand{\Comps}[1]{\widehat{#1}}
\newcommand{\wazs}{\bX}
\newcommand{\kc}{\kappa}
\newcommand{\pkc}{\kappa^*}
\newcommand{\kf}{\mathcal K}
\newcommand{\pkf}{{\mathcal K}^*}
\newcommand{\fullkf}{\overline{\mathcal{K}}}
\newcommand{\fullpkf}{\overline{\mathcal{K}}^*}
\newcommand{\betr}{B}
\newcommand{\centerf}{K}

\newcommand{\compf}{\Phi}
\newcommand{\rootf}{\rho}
\newcommand{\meet}{\mathbin{\wedge}}
\newcommand{\secf}{\sigma}
\newcommand{\nice}[2]{#1[#2]}
\newcommand{\subtree}[2]{#1(#2)}
\DeclareMathOperator{\Sam}{Sa}
\DeclareMathOperator{\MHP}{\cS}

\DeclareMathOperator{\Br}{Br}

\newcommand{\bxi}{\boldsymbol{\xi}}
\newcommand{\bps}{X} 
\newcommand{\cps}{\Comps{X}} 
\newcommand{\bp}{p} 
\newcommand{\bs}{s} 
\newcommand{\rR}{R} 
\newcommand{\cc}{\mathbf{c}} 
\newcommand{\rr}{\mathbf{r}} 
\newcommand{\CCLO}{\mathrm{CCLO}} 

\author{Gianluca Basso}
\address{Institut Camille Jordan \\
  Universit\'e Claude Bernard Lyon 1 \\
  Universit\'e de Lyon \\
  43, boulevard du 11 novembre 1918 \\
  69622 Villeurbanne \textsc{cedex} \\
  France}
\email{basso@math.univ-lyon-1.fr}

\author{Todor Tsankov}
\address{Institut Camille Jordan \\
  Universit\'e Claude Bernard Lyon 1 \\
  Universit\'e de Lyon \\
  43, boulevard du 11 novembre 1918 \\
  69622 Villeurbanne \textsc{cedex} \\
  France
  -- and --
  Institut Universitaire de France}
\email{tsankov@math.univ-lyon-1.fr}


\title{Topological dynamics of kaleidoscopic groups}

\begin{document}

\begin{abstract}
  Kaleidoscopic groups are a class of permutation groups recently introduced by Duchesne, Monod, and Wesolek. Starting with a permutation group $\Gamma$, the kaleidoscopic construction produces another permutation group $\mathcal{K}(\Gamma)$ which acts on a Wa\.{z}ewski dendrite (a densely branching tree-like compact space). 
  In this paper, we study how the topological dynamics of $\mathcal{K}(\Gamma)$ can be expressed in terms of the one of $\Gamma$, when the group $\Gamma$ is transitive. 
  By proving a Ramsey theorem for decorated rooted trees, we show that the universal minimal flow (UMF) of $\mathcal{K}(\Gamma)$ is metrizable iff $\Gamma$ is oligomorphic and the UMF of $\Gamma$ is metrizable. 
  More generally, we give concrete calculations, in an appropriate model-theoretic framework, of the UMF of $\mathcal{K}(\Gamma)$ when the UMF of a point stabilizer $\Gamma_c$ has a comeager orbit.
  Our results also give a large class of examples of non-metrizable UMFs with a comeager orbit. These results extend previous work of Kwiatkowska and Duchesne about the full homeomorphism groups.
\end{abstract}

\maketitle

\setcounter{tocdepth}{1}
\tableofcontents


\section{Introduction}
\label{sec:introduction}

Dendrites are one-dimensional, compact topological spaces that appear naturally in dynamics (for example, as Julia sets of quadratic polynomials). 
They are tree-like, but can have dense branching. This is the case for the universal spaces in this class, the \emph{\Wazewski dendrites} $W_n$ (where $n = 3, 4, \ldots, \infty$ is the order of the branch points), whose topology is also captured by \Fraisse limits of finite trees and whose rich homeomorphism groups have been studied in topology, in dynamics, and in model theory \cites{Charatonik1980,Kwiatkowska2018,Duchesne2019,Duchesne2020}. Even though the spaces $W_n$ are connected, their homeomorphism groups are \df{non-archimedean} (i.e., admit a basis at the identity consisting of open subgroups) and are best studied as permutation groups of the countable set of branch points of the dendrite.

\emph{Kaleidoscopic groups} are subgroups of $\Homeo(W_n)$ that were recently introduced by Duchesne, Monod, and Wesolek in \cite{Duchesne2019a}. For us a \df{permutation group $\Gamma \actson M$} is simply a \emph{closed} subgroup of the full symmetric group $\Sym(M)$, endowed with the topology of pointwise convergence. Starting with a permutation group $\Gamma \actson M$, where $M$ is a countable set with $n \geq 3$ elements, the kaleidoscopic construction produces another group $\kf(\Gamma) \leq \Homeo(W_n)$ whose \df{local action} is prescribed by $\Gamma$, that is, for any branch point $x \in W_n$, the action of the stabilizer $\kf(\Gamma)_x$ on the connected components of $W_n \sminus \set{x}$ is isomorphic to the action $\Gamma \actson M$. The groups $\kf(\Gamma)$ are Polish, non-archimedean, algebraically simple, and they have many other interesting properties, as discussed in \cite{Duchesne2019a}. They can also be represented as permutation groups acting on the branch points of the dendrite.
The construction was inspired by a classical paper of Burger and Mozes \cite{Burger2000} on groups acting on trees but in contrast with \cite{Burger2000}, kaleidoscopic groups are never locally compact because of the density of branch points. 

In this paper, we are interested in the topological dynamics of kaleidoscopic groups and, more precisely, in their continuous actions on compact spaces (or \df{flows}). In particular, we study which dynamical properties are preserved when passing from $\Gamma$ to $\kf(\Gamma)$. 
We recall that a flow is called \df{minimal} if every orbit is dense, and that, for any topological group $G$, there exists a \df{universal minimal flow (UMF)}, denoted by $\umf(G)$, characterized by the property that it equivariantly maps onto any other minimal $G$-flow.
$\umf(G)$ captures a fair amount of the dynamical information about $G$: the action on $\umf(G)$ is free whenever $G$ admits a free flow, and $\umf(G)$ is a singleton iff every $G$-flow has a fixed point, for instance. 
For locally compact, non-compact groups $G$, universal minimal flows are non-metrizable objects that are difficult to understand but, somewhat surprisingly, for large Polish groups, the UMFs often admit a concrete description and carry interesting combinatorial information. The landmark paper of Kechris, Pestov, and Todor\v{c}evi\'c \cite{Kechris2005} made explicit the connection between Ramsey theory and UMFs of non-archimedean Polish groups and motivated a lot of the recent work in structural Ramsey theory and in abstract topological dynamics. At the heart of the theory is the correspondence between the Ramsey property for the finite substructures of a given ultrahomogeneous structure $\bN$ and the fact that the automorphism group $\Aut(\bN)$ is \df{extremely amenable}, i.e., such that its UMF is a singleton.

In the recent investigations of minimal flows of Polish groups, two dividing lines have emerged. The better established one is metrizability: if $\umf(G)$ is metrizable, it is always the completion of a homogeneous space of the form $G/H$ for some closed, extremely amenable $H \leq G$ that is unique up to conjugation \cites{BenYaacov2017, Melleray2016}. In this case, one has a good understanding of the topological dynamics of $G$: for example, the minimal flows are all metrizable, they are concretely classifiable \cite{Melleray2016}, and for any $G$-flow $X$, the collection of almost periodic points in $X$ is closed \cite{Zucker2018a}.
Following \cite{Basso2021a}, we say that a Polish group is \df{CAP} if its UMF is metrizable. 
Notice that if $G$ is compact, then $\umf(G) \cong G$, so in particular it is metrizable.
Another important property is that if $G$ is CAP, then $\umf(G)$, and therefore all minimal $G$-flows, have a comeager orbit.

In fact, the existence of a comeager orbit in $\umf(G)$ provides the second interesting dividing line.
In \cite{Angel2014}, Polish groups with this property are said to have the \df{generic point property}. 
It is a tameness condition, which still implies a strong structure on the category of minimal $G$-flows, but is not completely understood at the moment.
Indeed, it was only recently that Kwiatkowska~\cite{Kwiatkowska2018} found the first examples of groups $G$ such that $\umf(G)$ has a comeager orbit but is not metrizable, namely the homeomorphism groups of certain \emph{generalized \Wazewski dendrites}.  
It is a consequence of our results about UMFs of kaleidoscopic groups that many of them have comeager orbits without being metrizable, suggesting that this phenomenon is more prevalent than previously thought.
Generalizing the metrizable case, Zucker~\cite{Zucker2021} showed that if $\umf(G)$ has a comeager orbit, then it is of the form $\Sam(G/H)$, where $H \le G$ is closed and extremely amenable, and $\Sam$ denotes the \df{Samuel compactification}. We refer to Section~\ref{sec:prel-topol-dynam} for all relevant definitions and more details.

The main results of this paper connect properties of the original permutation group $\Gamma$ and its UMF with properties of the UMF of the kaleidoscopic group $\kf(\Gamma)$. Prior work about the full homeomorphism group $G = \Homeo(W_n)$ had been done by Kwiatkowska~\cite{Kwiatkowska2018} (for all $n$) and Duchesne~\cite{Duchesne2020} (for $n = \infty$), who proved that $\umf(G)$ is metrizable and identified the appropriate subgroup $H \leq G$ such that $\umf(G) \cong \widehat{G/H}$. Duchesne also proposed (cf. \cite{Duchesne2020}*{Theorem~1.12}) a concrete description of $\umf(\Homeo(W_{\infty}))$ as a space of linear orders on the branch points of the dendrite, which is unfortunately not correct. We discuss this in detail in Section~\ref{sec:case-full-home}. In another paper \cite{Duchesne2021p}, Duchesne proved that the group $\kf(C_3)$, where $C_3$ is the cyclic group of order $3$ acting on itself by translations, also has a metrizable UMF. All of these previous results can be deduced from our general theorem below.

Our first result is a characterization of when $\umf(\kf(\Gamma))$ is metrizable. Recall that a permutation group $\Gamma \actson M$ is called \df{oligomorphic} if the diagonal action $\Gamma \actson M^k$ has finitely many orbits for all $k \in \N$.

\emph{Below and everywhere in the paper, when we consider a permutation group $\Gamma \actson M$, the set $M$ is tacitly assumed to be countable with $n \ge 3$ elements.} This is so that the corresponding \Wazewski dendrite and kaleidoscopic group $\kf(\Gamma)$ are well-defined.

\begin{theorem}[cf. \Cref{th:KGammaCAP}]
  \label{th:intro:KGammaCAP}
  Let $\Gamma \acts M$ be a transitive permutation group. Then the following are equivalent:
  \begin{enumerate}
    \item \label{i:intro:th:KGammaCAP-1} $\umf(\Gamma)$ is metrizable and the action $\Gamma \actson M$ is oligomorphic;
    \item \label{i:intro:th:KGammaCAP-2} $\umf(\kf(\Gamma))$ is metrizable.
  \end{enumerate}
\end{theorem}

Of course, item \ref{i:intro:th:KGammaCAP-1} above is trivially satisfied when $M$ is
finite, so we obtain the following.
\begin{cor}
  \label{c:intro:finite}
  Let $\Gamma \actson M$ be a transitive permutation group with $M$ finite. Then $\umf(\kf(\Gamma))$ is metrizable.
\end{cor}

As the action $\kf(\Gamma) \actson W_n$ is minimal, it is clear that if we want to search for an extremely amenable subgroup of $\kf(\Gamma)$, we have to look at subgroups of the stabilizers of the points in $W_n$. As the action $\kf(\Gamma) \actson W_n$ has a comeager orbit, namely that of the \emph{endpoints} (for the definition of the various types of points in $W_n$ and their properties, see Section~\ref{sec:kale-root-kale}), it is natural to study the stabilizer $\kf(\Gamma)_\xi$ for some fixed endpoint $\xi \in W_n$.

To that end, if $\Delta \actson M$ is a permutation group that stabilizes some distinguished point $c \in M$, we construct a \emph{root-kaleidoscopic group} $\pkf(\Delta)$ that is a subgroup of the stabilizer $\Homeo(W_n)_\xi$. The root-kaleidoscopic construction $\pkf(\Delta)$ has a similar behavior to $\kf(\Gamma)$: for any branch point $x \in W_n$, the local action of $\pkf(\Delta)_x$ on the connected components of $W_n \sminus \set{x}$ is isomorphic to $\Delta \actson M$ \emph{and the isomorphism is such that the component containing the root $\xi$ is always labeled by $c$}. 

This construction has the crucial property that if $\Gamma \actson M$ is transitive, then $\pkf(\Gamma_c) = \kf(\Gamma)_\xi$. 
An important insight of our work is that, therefore, in order to describe $\umf(\kf(\Gamma))$, for transitive $\Gamma$, one needs to understand the dynamics of the point stabilizers of the action $\Gamma \actson M$ rather than that of $\Gamma$ itself. 

Even though all results above are stated in the language of permutation groups, our combinatorial work is done in the model-theoretic setting of  ultrahomogeneous structures (which we recall in Section~\ref{sec:model-theory-permutation}).
We associate ultrahomogeneous structures $\bN$ and $\pkf(\bN)$ to $\Delta$ and $\pkf(\Delta)$, respectively.
The finite substructures of $\pkf(\bN)$ are rooted trees decorated by substructures of $\bN$. 

We then prove a transfer Ramsey theorem (\Cref{th:main-Ramsey}), which states that rooted trees decorated with Ramsey structures have the Ramsey property.
This generalizes a result of Soki\'{c} \cite{Sokic2015}, in which the decorations are linear orders. 
The proof goes by induction on the height of such trees, and uses a group-theoretic argument about permutational wreath products for the induction step (see \Cref{p:local-group-ext}).
\Cref{th:main-Ramsey} translates to the following statement about root-kaleidoscopic groups.

\begin{theorem}[cf. Corollary~\ref{cor:KstarDeltaExtAmenable}]
  \label{th:Intro:KstarDeltaExtAmenable}
  Let $\Delta \actson M$ be a permutation group, let $c \in M$, and suppose that $\Delta$ stabilizes $c$. If $\Delta$ is extremely amenable, then so is $\pkf(\Delta)$.
\end{theorem}

Using this theorem, together with the techniques from \cite{Kechris2005} and \cite{Zucker2021}, we obtain the following, from which Theorem~\ref{th:intro:KGammaCAP} can be deduced as a consequence.

\begin{theorem}[cf. Theorem~\ref{th:UMF-kal}]
  \label{th:intro:UMF-kal}
  Let $\Gamma \acts M$ be a transitive permutation group and let $c \in M$. If $\umf(\Gamma_c)$ has a comeager orbit, then so does $\umf(\kf(\Gamma))$. More precisely, if $\umf(\Gamma_c) = \Sam(\Gamma_c/\Delta)$ for some closed $\Delta \leq \Gamma_c$, then
  \begin{equation*}
    \umf(\kf(\Gamma)) = \Sam(\kf(\Gamma)/\pkf(\Delta)).
  \end{equation*}
\end{theorem}

In order to obtain a more concrete description of the Samuel compactification, we develop model-theoretic framework to represent it as an appropriate type space.
It encompasses the relational expansions, which are used in \cite{Kechris2005} to describe metrizable UMFs of automorphism groups.

As opposed to Theorem~\ref{th:intro:KGammaCAP}, which is a characterization, we do not have a converse of Theorem~\ref{th:intro:UMF-kal}. This is because it is not known whether the property of the UMF having a comeager orbit is closed under taking open subgroups (cf. Question~\ref{q:open-subgroup-of-comeager-orbit}).

Theorem~\ref{th:intro:UMF-kal} is already interesting when $\Gamma_c$ is the trivial group (so that one can take $\Delta = \Gamma_c$ above). Recall that a permutation action $\Gamma \actson M$ is called \df{simply transitive} if it is transitive and free (so that it is isomorphic to the left translation action of the group on itself).

\begin{cor}
  \label{c:intro:simply-transitive}
  Let $\Gamma \actson M$ be a simply transitive action. Then
  \begin{equation*}
    \umf(\kf(\Gamma)) = \Sam(\kf(\Gamma) / \kf(\Gamma)_\xi)
  \end{equation*}
  and $\umf(\kf(\Gamma))$ is metrizable iff $M$ is finite. In particular, if $M$ is infinite, $\umf(\kf(\Gamma))$  is a non-metrizable UMF with a comeager orbit.
\end{cor}

As a consequence of the results of \cite{Zucker2021}, another way to view the space $\Sam(\kf(\Gamma) / \kf(\Gamma)_\xi)$ that appears in the above corollary is as the \df{universal highly proximal extension} of the action $\kf(\Gamma) \actson W_n$. In Section~\ref{sec:maxim-highly-prox}, we give an explicit description of this flow as a type space in an appropriate model-theoretic structure.

Corollary~\ref{c:intro:simply-transitive} gives a large supply of non-metrizable UMFs with a comeager orbit and a fairly concrete description of them, given by Theorem~\ref{th:S_GW}. Continuum many such flows had been  exhibited by Kwiatkowska~\cite{Kwiatkowska2018}, as mentioned before. Her examples are parametrized by elements of the Cantor set, while Corollary~\ref{c:intro:simply-transitive} gives examples parametrized by countable groups, which, from descriptive-set theoretic point of view, are much more complicated.

In some cases, we can also calculate the \df{Furstenberg boundary} (or the \df{universal strongly proximal flow}) of $\kf(\Gamma)$. This result generalizes one of Duchesne~\cite{Duchesne2020} about the full homeomorphism group.
\begin{theorem}[cf. Theorem~\ref{th:furstenberg-boundary}]
  \label{th:intro:furstenberg-boundary}
  Let $\Gamma \acts M$ be a transitive permutation group and let $c \in M$. 
  If $\Gamma_c$ is amenable and $\umf(\Gamma_c)$ is metrizable, then the Furstenberg boundary of $\kf(\Gamma)$ is isomorphic to $\Sam(\kf(\Gamma) / \kf(\Gamma)_\xi)$.
\end{theorem}

Finally, we say a few words about our methods. It is well-known that model theory provides useful tools for studying oligomorphic permutation groups, because of Ryll-Nardzewski's theorem.
To aid with the study of a non-oligomorphic permutation group $\Gamma \acts M$, we develop the notion of a \df{full structure} for $\Gamma$ which associates to it an $\aleph_0$-categorical structure in a possibly uncountable language (cf. Theorem~\ref{th:aleph0-cat}) and which coincides with the classical construction in the oligomorphic case. While for model-theorists, $\aleph_0$-categorical structures in an uncountable language are probably just a curiosity, this viewpoint seems useful in dynamics: in particular, we obtain an interesting model-theoretic description of the universal highly proximal extension of the action $\kf(\Gamma) \actson W_n$ even when it is non-metrizable (cf. Theorem~\ref{th:S_GW}).

After some preliminary results in dynamics and model theory, most of the technical work in the paper is about constructing the root-kaleidoscopic structures, establishing the connection between them and point-stabilizers in kaleidoscopic groups (in Section~\ref{sec:kaleidoscopic_structures}), and proving the appropriate Ramsey theorem about them (Section~\ref{sec:ramsey-interlude}). The results about UMFs are collected and proved in Section~\ref{sec:univ-minim-flows}.

In Section~\ref{sec:case-full-home}, we apply the techniques we have developed to the full homeomorphism group $\Homeo(W_n)$ and study the relation of its UMF to the flow of \df{convex, converging linear orders (CCLO)} on the branch points introduced by Duchesne, which he proposed as the UMF of $\Homeo(W_\infty)$. More precisely, we prove the following.

\begin{theorem}[cf. \Cref{th:CCLO}]
  \label{th:intro:CCLO}
  Let $n = 3, 4, \ldots, \infty$, let $G = \Homeo(W_n)$, and let $\bps$ denote the set of branch points of $W_n$. Then the following hold:
  \begin{enumerate}
  \item  If $n$ is finite, then the flow $G \actson \CCLO(X)$ is isomorphic to $\umf(G)$;
  \item If $n = \infty$, then the flow $G \actson \CCLO(X)$ is a proper factor of $\umf(G)$. In particular, $\CCLO(X)$ is not isomorphic to $\umf(G)$.
  \end{enumerate}
\end{theorem}

\subsection*{Acknowledgments}
We are grateful to Andy Zucker for explaining to us the actual generality of Effros's theorem and its proof (cf. Fact~\ref{f:Effros}). We would also like to thank the anonymous referee for a careful reading of the paper and useful suggestions. Work on this project was partially supported by the ANR project AGRUME (ANR-17-CE40-0026) and the \emph{Investissements d'Avenir} program of Université de Lyon (ANR-16-IDEX-0005).


\section{Preliminaries on topological dynamics}
\label{sec:prel-topol-dynam}

Let $G$ be a topological group. A \df{$G$-flow} is a continuous action $G \actson X$ on a compact Hausdorff space $X$.
A $G$-flow is called \df{minimal} if every orbit is dense. A \df{factor map} between flows is a continuous, surjective, $G$-equivariant map. It is a classical fact, due to Ellis, that among all minimal $G$-flows there is a universal one that admits a factor map to any minimal $G$-flow and this property characterizes it uniquely up to isomorphism. It is called the \df{universal minimal flow} (\df{UMF} for short) of $G$ and we will denote it by $\umf(G)$. A group $G$ is \df{extremely amenable} if $\umf(G)$ is a singleton. 
Following \cite{Basso2021a}, we say that a Polish group $G$ is \df{CAP} if $\umf(G)$ is metrizable. (The name comes from a characterization of CAP groups: a group $G$ is CAP iff in every $G$-flow, the set of almost periodic points is closed. It turns out that this a more robust property that also makes sense for general topological groups.)

If $Y$ is a uniform space, we will denote by $\Sam(Y)$ the \df{Samuel compactification} of $Y$, i.e., the Gelfand space of the algebra $\UCB(Y)$ of all uniformly continuous, bounded functions on $Y$. It is well-known and easy to see that for a metrizable uniform space $Y$, $\Sam(Y)$ is metrizable iff $Y$ is precompact and in that case, $\Sam(Y)$ is simply the completion of $Y$.
The Samuel compactification has the following universal property: any uniformly continuous function $Y \to X$, with $X$ a compact Hausdorff space, extends (uniquely) to a continuous function $\Sam(Y) \to X$. 

Let $H \leq G$ be a closed subgroup of $G$. The homogeneous space $G/H$ carries a natural uniformity inherited from the right uniformity of $G$: a basis of entourages is given by sets of the form
\begin{equation*}
  \set{(gH, vgH) : v \in V, g \in G},
\end{equation*}
where $V$ varies over (a basis of) symmetric open neighborhoods of $1_G$.
In the case where $G$ is non-archimedean, it is convenient to restrict to open subgroups $V \leq G$ because then the sets above are equivalence relations whose classes are the $(V, H)$-double cosets. We will say that $H$ is \df{co-precompact} in $G$ if the space $G/H$ is precompact in this uniformity or, equivalently, if for every open $V \ni 1_G$, there exists a finite set $F \sub G$ such that $VFH = G$; if $G$ is non-archimedean, this just means that the set $\set{VgH : g \in G}$ of double $(V, H)$-cosets is finite for every open $V \leq G$. In particular, a permutation group $\Gamma \acts M$ is oligomorphic if and only if it is co-precompact in $\Sym(M)$. 
We say that $H$ is \df{presyndetic} if for every open $V \ni 1_G$, there exists a finite set $F \sub G$ such that $FVH = G$.

The Samuel compactification $\Sam(G/H)$ is always a $G$-flow. 
When $G$ is Polish, $\Sam(G/H)$ is metrizable iff $H$ is co-precompact and it is minimal iff $H$ is presyndetic (\cite{Zucker2021}*{Proposition~6.6}). Flows of the type $\Sam(G/H)$ appear naturally when one studies metrizable UMFs and UMFs with a comeager orbit. In the following fact, we collect the results that we will need.

\begin{fact}
  \label{f:facts-umf}
  Let $G$ be a Polish group. Then the following statements hold:
  \begin{enumerate}
  \item \label{i:f:umf:1} For any closed $H \leq G$, $\umf(G) = \Sam(G/H)$ iff $H$ is extremely amenable and presyndetic.
  \item \label{i:f:umf:2} $\umf(G)$ has a comeager orbit iff there exists $H \leq G$ extremely amenable and presyndetic. In that case, $\umf(G) = \Sam(G/H)$.
  \item \label{i:f:umf:3} $G$ is CAP iff there exists $H \leq G$ which is extremely amenable, co-precompact (and presyndetic).
  \end{enumerate}
\end{fact}
\begin{proof}
  \ref{i:f:umf:1}, \ref{i:f:umf:2} This follows from \cite{Zucker2021}*{Theorem~7.5}.

  \ref{i:f:umf:3} ($\Leftarrow$) is due to Pestov~\cite{Pestov2002a} but the argument is simple enough to include here.    The flow $\Sam(G/H)$ has the universal property that it maps to any $G$-flow with an $H$-fixed point. As $H$ is extremely amenable, every $G$-flow has an $H$-fixed point, so any minimal subflow of $\Sam(G/H)$ is isomorphic to $\umf(G)$. As $H$ is co-precompact in $G$, $\Sam(G/H)$ is metrizable and so is any subflow thereof.

  ($\Rightarrow$) is Theorem~1.1 in \cite{BenYaacov2017}.
\end{proof}

\begin{prop}
  \label{p:presyndetic-transitive}
  Let $G$ be a topological group and let $K \leq H \leq G$ be closed subgroups. If $H$ is presyndetic in $G$ and $K$ is presyndetic in $H$, then $K$ is presyndetic in $G$.
  Similarly, if we substitute co-precompact for presyndetic. 
\end{prop}
\begin{proof}
  We prove the statement for presyndeticity; the proof for co-precompactness follows the same lines.
  
  Let a neighborhood $V \ni 1_G$ be given and let $U \ni 1_G$ be open such that $U^2 \sub V$. Let $F_1 \sub H$ be finite such that $F_1 (U \cap H) K = H$. Let $U' = \bigcap_{f \in F_1} fUf^{-1}$ and let $F_2 \sub G$ be finite such that $F_2 U' H = G$. We claim that $(F_2F_1)VK = G$. Indeed, let $g \in G$. There exist $f_2 \in F_2, u_2 \in U'$, and $h \in H$ such that $g = f_2 u_2 h$. There are also $f_1 \in F_1, u_1 \in U$, and $k \in K$ such that $h = f_1 u_1 k$. Finally, we have that
  \begin{equation*}
    g = f_2 u_2 h = f_2 u_2 f_1 u_1 k = f_2 f_1 (f_1^{-1}u_2f_1) u_1 k \in F_2 F_1 U U K \sub (F_2F_1)VK,
  \end{equation*}
  as desired.
\end{proof}

It is easy to see that if $K \le H \le G$ are closed subgroups and $K$ is co-precompact in $G$, then $K$ is co-precompact in $H$. 
In particular, an oligomorphic subgroup of any permutation group is always co-precompact.

\begin{cor}
  \label{cor:umf-subgroups}
  Let $G$ be a Polish group. Then the following hold:
  \begin{enumerate}
  \item If $H \le G$ is a closed, presyndetic subgroup and $\umf(H) = \Sam(H/K)$ for some closed $K \leq H$, then $\umf(G) = \Sam(G/K)$.
  \item \label{i:cor:mf-subgroups-coprecompact} If $G$ has a co-precompact CAP subgroup, then $G$ is CAP.
  \end{enumerate}
\end{cor}

Spaces of the form $\Sam(G/H)$, where $G$ is non-archimedean, are not difficult to describe. We have the following.
\begin{fact}
  \label{f:SGH-non-archimedean}
  Let $G$ be a non-archimedean topological group and let $H \leq G$ be a closed subgroup. Then
  \begin{equation*}
    \Sam(G/H) = \varprojlim \beta (V \divby G / H),
  \end{equation*}
  where the inverse limit is taken over the collection of open subgroups of $G$ ordered by reverse inclusion and $\beta$ denotes the Stone--\v{C}ech compactification.
\end{fact}
\begin{proof}
  This follows from Gelfand duality and the fact that one can view the algebra $\ell^\infty(V \divby G / H)$ as a subalgebra of $\UCB(G/H)$ by identifying it with the $V$-invariant bounded functions and the union of all of these is dense in $\UCB(G/H)$.
\end{proof}

Effros's classical theorem from \cite{Effros1965} about non-meager orbits of Polish group actions is usually stated about actions on Polish spaces but it holds in considerably greater generality. We are grateful to Andy Zucker for explaining this to us. It is also used in the proof of \cite{Zucker2021}*{Theorem 5.5}.

\begin{fact}[Effros]
  \label{f:Effros}
  Let $X$ be a Hausdorff space. Let $G$ be a Polish group, let $G \actson X$ be a continuous action, and let $x_0 \in X$ be a point with a dense orbit. Then the following are equivalent:
  \begin{enumerate}
  \item \label{i:Effros:nonmeager} The orbit $G \cdot x_0$ is non-meager;
  \item \label{i:Effros:homeo} The map $G/G_{x_0} \to G \cdot x_0, gG_{x_0} \mapsto g \cdot x_0$ is a homeomorphism.
  \end{enumerate}
\end{fact}
\begin{proof}
  \begin{cycprf}
  \item[\impnext] This is Lemma~2.5 in \cite{Effros1965}. Effros's proof goes through in the generality of our statement, except that one needs to replace Lemma~2.4 with the following: for any open set $V \sub G$, the set $V \cdot x_0$ has the Baire property in $X$. This is a consequence of the following two facts. First, $V \cdot x_0$ is \df{analytic} (i.e., a continuous image of the space $\N^\N$), and second, every analytic set in a Hausdorff space has the Baire property. The latter follows from the classical theorem of Nikodým that the collection of sets with the Baire property is closed under the Souslin operation \cite{Kechris1995}*{29.14} and that in a Hausdorff space, analytic sets can be obtained via the Souslin operation from closed sets (see \cite{Bogachev2007}*{6.6.8}).
    
  \item[\impfirst] The hypothesis implies that $Y \coloneqq G \cdot x_0$ is Polish. Suppose that it is meager. Then there exist closed, nowhere dense subsets $F_n \sub X$, for $n \in \N$, with $Y \sub \bigcup_n F_n$. As $Y$ is non-meager in itself, there exist a non-empty open $U \sub X$ and $n \in \N$ such that $U \cap Y \sub F_n \cap Y$. But then $U \sub \cl{U \cap Y} \sub F_n$, contradicting the fact that $F_n$ is nowhere dense.
  \end{cycprf}
\end{proof}

If $X$ and $Y$ are topological spaces and $\pi \colon X \to Y$ is a surjective, continuous map, we will say that $\pi$ is \df{irreducible} if for every non-empty, open $U \sub X$, there exists a point $y \in Y$ with $\pi^{-1}(\set{y}) \sub U$. 
We say that $\pi$ is \df{almost 1-to-1} whenever the set $\set{x \in X : \pi^{-1}(\pi(x)) = \set{x}}$ is comeager. 
If $\pi$ is onto and almost 1-to-1, then it is irreducible, and the converse holds if $X$ is metrizable and Baire. See \cite{deVries2014}*{A.9} for more details on irreducible maps.

If $\pi \colon X \to Y$ is a factor map between $G$-flows which is irreducible, we will say that $X$ is a \df{highly proximal extension of $Y$}. It is a result of Auslander and Glasner~\cite{Auslander1977}, for minimal flows, and of Zucker~\cite{Zucker2018} in general, that among all highly proximal extensions of a given flow $X$, there exists a universal one, which we will denote by $\MHP_G(X) \to X$. It has the property that for any other highly proximal extension $X' \to X$, there is a (highly proximal) map $\MHP_G(X) \to X'$ such that the appropriate diagram commutes.
An important feature of highly proximal extensions is that they preserve many dynamical properties.

\begin{fact}[\cite{Auslander1977}]
  Let $\pi \colon X \to Y$ be a highly proximal factor map between $G$-flows with $Y$ minimal. Then $X$ is minimal and if $Y$ is proximal (strongly proximal), then so is $X$. 
\end{fact}  

A flow $X$ such that $\MHP_G(X) = X$ is called \df{maximally highly proximal}, or \df{MHP}. Maximally highly proximal flows were extensively studied by Zucker~\cite{Zucker2021}. 
If $H \leq G$ is a closed subgroup of $G$, then the flow $\Sam(G/H)$ is always MHP (see \cite{Zucker2021}*{Subsection~3.2.1}). There is an important connection between Samuel compactifications of homogeneous spaces and universal highly proximal extensions of flows with a comeager orbit.

\begin{prop}
  \label{p:SG-comeager}
  Let $G$ be a Polish group and let $G \actson X$ be a $G$-flow with a comeager orbit $G \cdot x_0$. Then:
  \begin{enumerate}
    \item \label{i:f:SG-comeager:1} $\MHP_G(X) \cong \Sam(G/G_{x_0})$;
    \item \label{i:f:SG-comeager:2} Any highly proximal extension of $X$ is almost 1-to-1;
    \item \label{i:f:SG-comeager:3} If $X$ is moreover minimal, then for any two factor maps $\pi, \pi' \colon \MHP_G(X) \to X$, there is an automorphism $f$ of $\MHP_G(X)$ such that $\pi \circ f = \pi'$. In particular, any factor map $\MHP_G(X) \to X$ is almost 1-to-1.
  \end{enumerate}
\end{prop}
\begin{proof}
  Consider the natural bijection $\rho \colon G/G_{x_0} \to G \cdot x_0$, which by Effros's theorem (Fact~\ref{f:Effros}) is a  homeomorphism.
  By the universal property of $\Sam(G/G_{x_0})$, $\rho$ extends to a map $\pi \colon \Sam(G/G_{x_0}) \to X$. 
  We prove that $\pi^{-1}(x_0) = \set{G_{x_0}}$. 
  Indeed, suppose that $p\in \Sam(G/G_{x_0})$ is such that $\pi(p) = x_0$.
  Since $G/G_{x_0}$ embeds densely in $\Sam(G/G_{x_0})$, there is a net $(g_i G_{x_0})$ of elements of $G/G_{x_0}$ which converges to $p$. 
  By continuity, $\pi(g_i G_{x_0}) =  \rho(g_i G_{x_0})= g_i \cdot x_0$ converges to $x_0 = \rho(G_{x_0})$.
  But $\rho$ is a homeomorphism, so $g_i G_{x_0} \to G_{x_0}$, that is, $p = G_{x_0}$. 

  By \cite{Engelking1989}*{Theorem 4.3.26}, since $G/G_{x_0}$ is Polish, it is a $G_\delta$ set in all of its compactifications, in particular in $\Sam(G/G_{x_0})$. 
  Therefore, $\pi$ is almost 1-to-1, as $\pi^{-1}(\pi(g G_{x_0})) = g \pi^{-1}(x_0) = \set{g G_{x_0}}$ for all $g G_{x_0} \in G/G_{x_0}$. 

  Now \ref{i:f:SG-comeager:1} follows from the universal property of $\MHP_G(X)$: since $\pi$ is almost 1-to-1, it is highly proximal, so $\MHP_G(X)$ is a highly proximal extension of $\Sam(G/G_{x_0})$. As $\Sam(G/G_{x_0})$ is MHP, this extension must be an isomorphism.

  \ref{i:f:SG-comeager:2} Since the universal highly proximal extension is almost 1-to-1, so is any other highly proximal extension.

  \ref{i:f:SG-comeager:3} If $X$ is minimal, then $\MHP_G(X) \cong \Sam(G/G_{x_0})$ is minimal, too. Let $\pi' \colon \MHP_G(X) \to X$ be any factor map. 
  By Fact~\ref{f:Effros}, $G \cdot x_0$ is Polish, so by \cite{Engelking1989}*{Theorem 4.3.26}, it is $G_\delta$ in $X$, so $\pi'^{-1}(G \cdot x_0)$ is  $G_\delta$ and dense by minimality, so it must intersect the dense $G_\delta$ set $G/G_{x_0}$ in $\Sam(G/G_{x_0})$. In particular, $\pi'[G/G_{x_0}] = G \cdot x_0$.

  Then $\phi \coloneqq \rho^{-1} \pi' |_{G/G_{x_0}}$ is a $G$-equivariant bijection of the homogeneous space $G/G_{x_0}$, so there exists an element $h$ in the normalizer of $G_{x_0}$ in $G$ such that $\phi$ is left multiplication by $h$. In particular, $\phi$ is a $G$-equivariant uniform isomorphism of $G/G_{x_0}$, so it extends to an automorphism $f$ of the flow $\Sam(G/G_{x_0})$. If we define $\pi$ as in \ref{i:f:SG-comeager:1}, we get that $\pi \circ f = \pi'$ as desired.
\end{proof}

We finally record some of the closure properties of relevant classes of topological groups. Recall that $G$ is called \df{amenable} if every $G$-flow carries an invariant probability measure.
\begin{fact}
    \label{f:ea-open-subgroup}
    Let $G$ be a Polish group and let $V \le G$ be an open subgroup. Let $K$ be a topological group and let $\pi \colon G \to K$ be a continuous homomorphism with dense image. If $G$ is amenable/extremely amenable/CAP, then so are $V$ and $K$.
\end{fact}
\begin{proof}
  All of these are standard. For $K$, it is a direct consequence of the definitions. For $V$, it is possible to give a uniform proof using the co-induction construction as follows. If we start with any flow $V \actson X$, then one can define an action of $G$ on the space $X^{G/V}$ and a $V$-equivariant map $X^{G/V} \to X$ (see, e.g., \cite{Bodirsky2013}*{Lemma~13}).
  Now it is clear that if $X^{G/V}$ has an invariant measure/fixed point/metrizable subflow, then so does $X$.
\end{proof}

\begin{question}
  \label{q:open-subgroup-of-comeager-orbit}
    If $G$ is Polish, $\umf(G)$ has a comeager orbit and $V \le G$ is an open subgroup, is it true that $\umf(V)$ has a comeager orbit?
\end{question}
  

\section{The model theory of permutation groups}
\label{sec:model-theory-permutation}
  
Even though the Wa\.{z}ewski dendrites are connected spaces, their homeomorphism groups (and subgroups thereof) are non-archimedean and are best studied as permutation groups. An important tool for understanding infinite permutation groups is \Fraisse theory, where the group is viewed as the automorphism group of an \emph{ultrahomogeneous structure} and one can study it via the combinatorial properties of the \emph{age} of the structure. This is particularly helpful in topological dynamics, where one can use the Kechris--Pestov--Todor\v{c}evi\'c correspondence to study extreme amenability and describe the universal minimal flows of automorphism groups.


\subsection{\Fraisse limits and full structures}
\label{sec:fraisse-limits-full}

We recall all relevant definitions and establish some terminology and notation. Most of them are standard except perhaps for our treatment of expansions. A \df{signature} $\cL$ is a collection of relation symbols $\set{R_i}_i$ and function symbols $\set{F_j}_j$ to each of which is associated a natural number called its \df{arity}. An $\cL$-structure $\bM$ consists of a non-empty set $M$, called \df{the universe of $\bM$}, together with interpretations of the relation and function symbols. Each relation symbol of arity $k$ is interpreted as a subset of $M^k$ and each function symbol of arity $k$ is interpreted as a function $M^k \to M$. A function symbol of arity $0$ is called a \df{constant symbol}. A signature without function symbols is called \df{relational}. An \df{automorphism} of $\bM$ is an element of $\Sym(M)$ which respects all relation and function symbols. A \df{substructure} $\bM'$ of $\bM$ is given by a subset $M' \sub M$ which is closed under the functions in the signature (in particular, it contains all constants). The interpretations of the symbols in $\bM'$ are inherited from $\bM$. An \df{embedding} $\bM \to \bM'$ is an isomorphism from $\bM$ onto a substructure of $\bM'$. A structure $\bM$ is called \df{ultrahomogeneous} if every isomorphism between finitely generated substructures of $\bM$ extends to an automorphism of $\bM$. In this paper, we will only be interested in countable structures (but in possibly uncountable signatures).

The \df{age} of a countable structure $\bM$ is given by the isomorphism classes of all finitely generated substructures of $\bM$. Every age is countable, \df{hereditary} (closed under taking substructures) and \df{directed} (every two structures in the age embed into a third). An age is simply a collection of isomorphism classes of finitely generated structures with these three properties. It is easy to see that every age is the age of some countable structure. If $\cF$ is an age, we will say that $\bM$ is an \df{$\cF$-structure} if $\Age(\bM) \sub \cF$.

\Fraisse isolated an important additional property, called \df{amalgamation}, which characterizes the ages of ultrahomogeneous structures, and proved that for every age with amalgamation, there exists a unique countable, ultrahomogeneous structure with this age, up to isomorphism (often called its \df{\Fraisse limit}).

Let $\cL$ be a signature. An \df{atomic formula} in the variables $\bar v = (v_1, \ldots, v_n)$ is of the form $R(t_1(\bar v), \ldots, t_k(\bar v))$, where $R$ is a relation symbol and $t_1, \ldots, t_k$ are \df{terms}, i.e., compositions of functions symbols. A \df{quantifier-free formula} is a Boolean combination of atomic formulas. A \df{first-order formula} (or just a \df{formula}) is one where we also allow quantifiers. Note that all quantifier-free formulas are preserved by embeddings but first-order formulas in general are not. For the purposes of \Fraisse theory, we may and will simply consider quantifier-free formulas as part of the signature. If $\theta(\bar v)$ is a formula and $\bM$ is a structure, we will denote by $\theta(\bM)$ the set $\set{\bar a \in M^{\bar v} : \bM \models \theta(\bar a)}$. Sets of this form are called \df{definable} (or \df{$0$-definable} when we want to emphasize the absence of parameters).
For a fixed $B \sub \bM$, we can consider the subsets of $M^k$ which are \df{definable with parameters from $B$}, or \df{$B$-definable}, i.e., of the form $\set{\bar a \in M^{\bar v} : \bM \models \phi(\bar a, \bar b)}$ for some formula $\phi(\bar v, \bar w)$ and $\bar b \in B^{\bar w}$. A structure \df{eliminates quantifiers} if every $0$-definable set is definable by a quantifier-free formula.

The \df{theory of $\bM$}, denoted by $\Th(\bM)$, is the collection of all sentences true in $\bM$.
Recall that the space of \df{$k$-types} (for $k \in \N$) of $\Th(\bM)$, denoted by $\tS_k(\Th(\bM))$, is the Stone space of the Boolean algebra of $0$-definable subsets of $M^k$. For $A \sub M$, we can also similarly consider the type space with parameters from $A$, denoted by $\tS_k(A)$. We also let $\tS_\omega = \varprojlim \tS_k$. These are all $\aut(\bM)$-flows (if the set of parameters is $\Aut(\bM)$-invariant). They are metrizable if the signature is countable.
For a tuple $\bar a \in M^k$, the \df{type of $\bar a$}, denoted by $\tp \bar a$ is the element $p \in \tS_{k}(\Th(\bM))$ given by
\begin{equation*}
  p(\bar v) \models \phi(\bar v) \iff M \models \phi(\bar a), \quad \text{ for every formula } \phi.
\end{equation*}
We say that such types are \df{realized} in $\bM$. Any type $p \in \tS_{k}(\Th(\bM))$ is realized in some model of $\Th(\bM)$.

We will at some point use quantifiers, so we will also need the following weaker notion of homogeneity: a structure $\bM$ is \df{homogeneous} if for all tuples $\bar a$, $\bar b$ from $M$ such that $\tp \bar a = \tp \bar b$, there exists $g \in \Aut(\bM)$ such that $g \cdot \bar a = \bar b$. Of course, if a model is ultrahomogeneous, then it is homogeneous. A type is \df{isolated} if it is an isolated point in the corresponding type space. A model that only realizes isolated types is called \df{atomic}. It is a standard fact that a countable, atomic model is always homogeneous.

Let now $\cF$ be an age in the signature $\cL$. An \df{imaginary sort} $\cS$ for $\cF$ is given by a pair of quantifier-free formulas $\theta(\bar v)$, $\eps(\bar u_1, \bar u_2)$ such that $|\bar v| = |\bar u_1| = |\bar u_2|$ and for every $\bA \in \cF$, $\eps(\bA) \cap (\theta(\bA) \times \theta(\bA))$ is an equivalence relation on $\theta(\bA)$. We define $\cS(\bA)$ to be the set of classes of this equivalence relation.

The following lemma records some basic facts about imaginary sorts and we omit the routine proof.
\begin{lemma}
  \label{l:basic-imaginary-sorts}
  Let $\cF$ be an age and let $\cS$ be an imaginary sort for $\cF$ given by $(\theta, \eps)$. Then the following hold:
  \begin{enumerate}
  \item For any $\cF$-structure $\bM$, $\eps(\bM) \cap (\theta(\bM) \times \theta(\bM))$ is an equivalence relation on $\theta(\bM)$ and, in particular, we can define
    \begin{equation*}
      \cS(\bM) = \set{[\bar a]_{\eps(\bM)} : \bar a \in \theta(\bM)}.
    \end{equation*}
  \item If $f \colon \bM \to \bN$ is an embedding between $\cF$-structures, then $f$ naturally defines an injective map $\cS(f) \colon \cS(\bM) \to \cS(\bN)$ by
    \begin{equation*}
      \cS(f)([\bar a]_{\eps(\bM)}) = [f(\bar a)]_{\eps(\bN)} \quad \text{ for } \bar a \in \theta(\bM).
    \end{equation*}
    Moreover, this correspondence is functorial.
  \item In particular, for any $\cF$-structure $\bM$, the group $\Aut(\bM)$ acts on $\cS(\bM)$.
  \end{enumerate}
\end{lemma}

We will also allow to have symbols from the signature defined on imaginary sorts. This will be particularly useful for defining expansions.

If one is given a set $M$ and a permutation group $\Gamma \leq \Sym(M)$, one can define a relational, ultrahomogeneous structure $\bM$ with universe $M$ such that $\Aut(\bM) = \Gamma$. There are many such structures and all of them give rise to equivalent \Fraisse classes, however they do differ from the point of view of model theory. First, we always make the assumption that $\bM$ is atomic: equivalently, for every $k$ and every $\Gamma$-orbit $O \sub M^k$, there exists a quantifier-free formula $\theta(\bar u)$ such that $\theta(M) = O$. We will say that the formula $\theta$ \df{isolates} the orbit $O$. With this assumption, there are two canonical choices corresponding to the minimal and the maximal possible Boolean algebras of definable sets, respectively, under the constraint that the structure is ultrahomogeneous and $\Aut(\bM) = \Gamma$. The minimal one is obtained by introducing a relation symbol for every $\Gamma$-orbit on $M^k$ for every $k$ and it has the advantage that the resulting signature is always countable; however, for us, it is the other one which will be more relevant.
\begin{defn}
  \label{df:full-structure}
  Let $M$ be a countable set and let $\Gamma \leq \Sym(M)$ be a closed subgroup. We will say that a structure $\bM$ with universe $M$ is \df{full for $\Gamma$} if $\Aut(\bM) = \Gamma$ and for every $k \in \N$, every $\Gamma$-invariant subset of $M^k$ is definable in $\bM$. We will say it is \df{qf-full} if every such set is quantifier-free definable. A structure is \df{full} if it is full for its automorphism group.
\end{defn}

Note if $\Gamma \actson M$ is a permutation group, then all full structures for $\Gamma$ are interdefinable. If $\Gamma$ is oligomorphic, then every structure $\bM$ with automorphism group $\Gamma$ is full; if moreover $\bM$ is ultrahomogeneous, then it is qf-full. If $\Gamma$ is not oligomorphic, then every full structure necessarily has a signature of cardinality continuum. As we will see shortly, a full structure is always \df{$\aleph_0$-categorical} (i.e, the unique countable model of its first-order theory), which makes it a canonical choice for studying permutation groups model-theoretically.

\begin{theorem}
  \label{th:aleph0-cat}
  Let $M$ be a countable set and let $\Gamma \leq \Sym(M)$ be a closed subgroup. Let $\bM$ be a full structure for $\Gamma$ with universe $M$ and let $T = \Th(\bM)$. Then the following hold:
  \begin{enumerate}
  \item \label{i:aleph0:types} $\tS_k(T) = \beta(M^k/\Gamma)$;
  \item \label{i:aleph0:atomic} $\bM$ is an atomic model of $T$;
  \item \label{i:aleph0:categ} The structure $\bM$ is $\aleph_0$-categorical;
  \item \label{i:aleph0:qfelim} If $\bM$ is qf-full, then $\bM$ is ultrahomogeneous and eliminates quantifiers.
  \end{enumerate}
\end{theorem}
\begin{proof}
  \ref{i:aleph0:types} A subset of $M^k$ is definable iff it is $\Gamma$-invariant and the Boolean algebra of $\Gamma$-invariant subsets of $M^k$ is isomorphic to the power set of $M^k / \Gamma$.

  \ref{i:aleph0:atomic} The types realized in $\bM$ correspond to the principal ultrafilters in $\beta(M^k / \Gamma)$. As the principal ultrafilters are isolated, this implies that $\bM$ is atomic.

  \ref{i:aleph0:categ} We will show that any countable model of $T$ is atomic. Then the claim will follow from the uniqueness of the countable atomic model. The proof is inspired by a construction of Blass~\cite{BlassMathOverflow}.

  Suppose that $q \in \tS_k(T)$ is a non-isolated type; our goal is to show that any model that realizes $q$ is uncountable. Let $p_0, p_1, \ldots$ enumerate the isolated types in $\tS_k(T)$ and let $\phi_i$ be a formula that isolates $p_i$. Let $\set{B_\alpha : \alpha \in 2^\N}$ be a collection of infinite subsets of $\N$ such that for all $\alpha \neq \alpha'$, $B_\alpha \cap B_{\alpha'}$ is finite. Let $\psi_\alpha(x)$ be the relation defined in $\bM$ by $\bigvee_{i \in B_\alpha} \phi_i(x)$ and let  $<$ be the relation defined in $\bM$ by
  \begin{equation*}
    x < y \iff \phi_i(x) \land \phi_j(y) \land i < j.
  \end{equation*}
  Let now $c$ be a realization of $q$ in some model $\bN$ and consider the collection of formulas
  \begin{equation*}
    \set{\exists x \ \psi_\alpha(x) \land x > c : \alpha \in 2^\N}.
  \end{equation*}
  We claim that each one of them holds and their solutions are disjoint, showing that $\bN$ is uncountable. The first claim follows from the fact that the sentence $\forall y \exists x \ \psi_\alpha(x) \land x > y$ holds in $\bM$ (because $B_\alpha$ is infinite). For the second, note that for $\alpha \neq \alpha'$, the sentence
  \begin{equation*}
    \forall x \ \psi_\alpha(x) \land \psi_{\alpha'}(x) \implies \bigvee_{i \in B_\alpha \cap B_{\alpha'}} \phi_i(x)
  \end{equation*}
  is true in $\bM$ and the fact that $q$ is non-principal implies that $c > y$ for every $y$ for which there is $i$ such that $\phi_i(y)$.

  \ref{i:aleph0:qfelim} This follows from the fact that all definable sets in $\bM$ are $\Gamma$-invariant, so quantifier-free definable.
\end{proof}


\subsection{Expansions as types}
\label{sec:expansions-types}

Starting from \cite{Kechris2005}, metrizable universal minimal flows of automorphism groups of a countable structure $\bM$ are often represented as spaces of expansions of $\bM$ in a suitable relational signature. In principle, this can always be done and it works well when the expansion is simple (say, by a linear ordering) but in more complicated situations, a more expressive formalism may be convenient. For example, in order to describe the UMF of kaleidoscopic groups, we are naturally led to consider expansions by a constant not belonging to the original structure, which, while possible to deal with in the traditional KPT setting (see \cite{Kwiatkowska2018}), require a lot of coding and are less transparent.

Our definition of expansions is more general than the usual one on two counts: first, we allow expansions by function symbols (in particular, constants) that may add new elements to the structure, and second, we also allow new symbols on the imaginary sorts. Formally, let $\bM$ be an $\cL$-structure and let $\cL' \supseteq \cL$. An \df{$\cL'$-expansion} of $\bM$ is an $\cL'$-structure $\bN$ together with a map $\iota \colon M \to N$ such that:
\begin{itemize}

\item $\iota$ is an $\cL$-embedding;
\item $N = \dcl^{\cL'} \iota[M]$;
\item $\iota[M]$ is $\Aut(\bN)$-invariant.
\end{itemize}
Moreover, we allow some of the new symbols in $\cL'$ to apply to the imaginary sorts of $\bN$. As usual, $\dcl$ denotes the \df{definable closure}, i.e., $\dcl A$ is the collection of elements $b \in N$ such that the singleton $\set{b}$ is $A$-definable. 
The main property that our definition ensures is that an expansion gives rise to a closed subgroup of the automorphism group, which, as the proof of \Cref{p:SGH-as-types} shows, is all that matters.

\begin{prop}
  Let $\bM$ be an $\cL$-structure and $\bN$ be a $\cL'$-expansion of $\bM$ with embedding $\iota \colon M \to N$.
  Then the restriction map $f \colon \Aut(\bN) \to \Aut(\bM)$ defined by $f(g) = g|_{\iota[M]}$ is a topological group embedding.
\end{prop}
\begin{proof}
  We may assume that $\iota$ is the identity, so that $\bM$ is a substructure of the $\cL$-reduct of $\bN$. 
  Since $M$ is $\Aut(\bN)$-invariant, $g|_{M} \in \aut(\bM)$.

  Suppose that $f(g) = \id_M$. 
  Then $g(b) = b$ for any $b \in \dcl(M) = N$, so $g = \id_N$.
  
  To show that $f$ is an embedding, let $(g_n)$ be a sequence of elements of $\Aut(\bN)$ such that $g_n |_M \to \id_M$.
  Fix $b \in N$ in order to prove that eventually $g_n(b)= b$.
  Let $\phi(v, \bar a)$ define $\set{b}$; then for any $h \in \Aut(\bM)$, which fixes $\bar a$, we have $h(b) = b$. 
  But $g_n \to \id_M$, so eventually $g_n \cdot \bar a = \bar a$ and we are done.
\end{proof}

We say that an expansion is \df{relational} if $N = M$ and $\cL' \setminus \cL$ is relational. 

Next we explain how to represent any flow of the form $\Sam(G/H)$, where $G = \Aut(\bM)$ and $H \leq G$ is an arbitrary closed subgroup, as a space of types for an appropriate expansion. 
The following is just a restatement of Fact~\ref{f:SGH-non-archimedean} in a model-theoretic setting but being able to use arbitrary first order formulas is convenient in applications and also makes the action of the group $G$ more explicit.

If $\bN$ is an expansion of $\bM$, we denote
\begin{equation*}
  \Exp(\bM, \bN) = \set{p \in \tS_\omega(\Th(\bN)) : p|_\cL = \tp_\cL \bar a},
\end{equation*}
where $\bar a$ is some fixed enumeration of $\bM$. It will also often be convenient to identify the variables of the type with the elements of $M$ via this fixed enumeration. Note that $\Aut(\bM)$ acts naturally on $\Exp(\bM, \bN)$ by permuting the variables.
\begin{prop}
  \label{p:SGH-as-types}
  Let $\bM$ be a countable structure in a signature $\cL$, which is an atomic model of its first-order theory. Suppose that $\bN$ is an expansion of $\bM$ in a signature $\cL'$, and that $\bN$ is full. Denote $G = \Aut(\bM)$ and $H = \Aut(\bN)$. Then
  \begin{equation*}
    \Sam(G/H) \cong \Exp(\bM, \bN).
  \end{equation*}
\end{prop}
\begin{proof}
  From Fact~\ref{f:SGH-non-archimedean}, we have that
  \begin{equation*}
    \Sam(G/H) = \varprojlim \beta (V \divby G / H) = \varprojlim \beta (G_{\bar b} \divby G / H),
  \end{equation*}
  where the second limit is taken over all finite tuples $\bar b$ in $M$. The double coset space $G_{\bar b} \divby G / H$ can be identified with the set of orbits of the action $H \actson G \cdot \bar b$. Now, since $\bM$ is atomic, it is homogeneous, so 
  \begin{equation*}
    G \cdot \bar b = \set{\bar c \in M^{|\bar b|} : \tp_\cL \bar c = \tp_\cL \bar b}
  \end{equation*}
  is an $\cL$-definable set and by the fullness of $\bN$, every $H$-invariant subset of $G \cdot \bar b$ is $\cL'$-definable, so we obtain a natural identification between the power set of $G_{\bar b} \divby G / H$ and the algebra of $\cL'$-definable subsets of $G \cdot \bar b$. Passing to the respective Stone spaces, we obtain the identification
  \begin{equation*}
    \beta(G_{\bar b} \divby G / H) = \set{p \in \tS_{|\bar b|}(\Th(\bN)) : p|_\cL = \tp_\cL \bar b}.
  \end{equation*}
Finally, taking the inverse limit over all finite tuples $\bar b$ in $\bM$ completes the proof.
\end{proof}

As a simple illustration, we show how to represent model-theoretically the flow $\Sam(G/V)$ for an open subgroup $V \leq G$.
\begin{cor}
  \label{p:samuel-open}
  Let $\Gamma \leq \Sym(M)$ be a transitive permutation group and let $\bM$ be a full structure for $\Gamma$. Let $c \in M$. Then the following hold:
  \begin{enumerate}
  \item \label{i:pso:flow} $\Sam(\Gamma/\Gamma_c) \cong \tS_1(M)$ as $\Gamma$-flows;
  \item \label{i:pso:metrizable} $\Sam(\Gamma/\Gamma_c)$ is metrizable iff $\Gamma$ is oligomorphic.
  \end{enumerate}
\end{cor}
\begin{proof}
  \ref{i:pso:flow} Apply Proposition~\ref{p:SGH-as-types} to the expansion by a single constant symbol $\cc$ which is interpreted as an element $c \in M$ (it does not matter which one because of transitivity). 
  This expansion is a full structure for $\Gamma_c$ since any $\Gamma_c$-invariant set $D \sub M^k$ is defined by $\phi(\cc, \bar v)$, where $\phi(u, \bar v)$ is a formula defining $\Gamma \cdot (\set{c} \times D)$ in $\bM$.

  \ref{i:pso:metrizable} ($\Rightarrow$) Suppose that $\Gamma \actson M$ is not oligomorphic. Then there exists a tuple $\bar a$ in $M$ such that $\Gamma_{\bar a}$ has infinitely many orbits on $M$. Reasoning as above, we see that the expansion of $\bM$ by constants for $\bar a$ remains full. Therefore, by Theorem~\ref{th:aleph0-cat}, $\tS_1(\bar a) \cong \beta(M/\Gamma_{\bar a})$, which is not metrizable. So $\tS_1(M)$, which surjects onto $\tS_1(\bar a)$, is not metrizable either.

  ($\Leftarrow$) If $\Gamma$ is oligomorphic, then $\Gamma_c$ is also oligomorphic and in particular, $\Gamma_c$ is co-precompact in $\Gamma$. Therefore $\Sam(\Gamma/\Gamma_c)$ is metrizable.
\end{proof}

We recall a combinatorial criterion for presyndeticity. If $\bM$ is an ultrahomogeneous structure and $\bN$ is an ultrahomogeneous, relational expansion of $\bM$, we will say that $\bN$ is a \df{minimal expansion} of $\bM$ if for every $\bB' \in \Age(\bN)$, there exists $\bA \in \Age(\bM)$ such that for every $\bA' \in \Age(\bN)$ which is an expansion of  $\bA$, there exists an embedding of $\bB'$ into $\bA'$. 
Zucker considered this condition in a slightly different context in \cite{Zucker2018a}. 
A similar property had already been studied in the literature under the name of the \df{ordering property} (see, e.g., \cite{Kechris2005}), when the expansion is by a linear order, or the \df{expansion property} for general relational expansions (see \cite{NguyenVanThe2013}). 
The definition above is equivalent to the expansion property for a precompact expansion but not in general. The following is well-known and implied in \cite{Zucker2018a} but we could not find a direct reference, so we include the proof of the direction we need.
\begin{prop}
  \label{p:presyndetic-criterion}
  Let $\bM$ be a countable, ultrahomogeneous structure and let $\bN$ be an ultrahomogeneous, relational expansion of $\bM$. Then the following are equivalent:
  \begin{enumerate}
  \item $\bN$ is a minimal expansion of $\bM$;
  \item $\Aut(\bN)$ is presyndetic in $\Aut(\bM)$.
  \end{enumerate}
\end{prop}
\begin{proof}
  \begin{cycprf}
  \item[\impnext] Write $G = \Aut(\bM)$ and $H = \Aut(\bN)$. Let $\bar b$ be a finite tuple in $M$ and denote by $\bB$ and $\bB'$ the substructure generated by $\bar b$ in $\bM$ and in $\bN$, respectively. (Note that because the expansion is relational, $\bB$ and $\bB'$ are structures with the same underlying set $B$.) Let $\bA \in \Age(\bM)$ be a witness of minimality for $\bB'$ and realize $\bA$ as a substructure of $\bM$. Let $F$ be a finite subset of $G$ such that $\binom{\bA}{\bB} = \set{f|_{B} : f \in F}$. Our goal is to show that $G = HG_BF^{-1}$. To that end, let $g \in G$ be arbitrary. We want to find $h \in H$ and $f \in F$ such that $hgf$ fixes $B$. Let $\bA'$ be the substructure of $\bN$ with underlying set $g \cdot A$. By minimality of the expansion and the choice of $F$, there exists $f \in F$ such that $gf \cdot \bB' \cong \bB'$. Now we conclude by ultrahomogeneity of $\bN$.
    
  \item[\impfirst] The proof is similar and we omit it.
  \end{cycprf}
\end{proof}


\section{Kaleidoscopic and root-kaleidoscopic groups}
\label{sec:kale-root-kale}


\subsection{\Wazewski dendrites and kaleidoscopic groups}
\label{sec:wazewski-dendr-kale}

An important property of topological trees is that they are \df{uniquely arcwise connected}, i.e., between every two points there is a unique arc connecting them. (An \df{arc} is just a homeomorphic copy of the unit interval $[0, 1]$.) Dendrites are topological spaces, generalizing trees, where branch points are allowed to be dense. More formally, a \df{dendrite} is a compact, metrizable, locally connected, uniquely arcwise connected, topological space.

Let $D$ be a dendrite and for a point $x \in D$, denote by $\comps{x}$ the set of connected components of $D \sminus \set{x}$. The points of $D$ are classified by the cardinality of $\comps{x}$: $x$ is called an \df{endpoint} if $\abs{\comps{x}} = 1$; it is called a \df{regular point} if $\abs{\comps{x}} = 2$; and it is called a \df{branch point} if $\abs{\comps{x}} \geq 3$. If $x$ is a branch point, the cardinal $\abs{\comps{x}}$ is called the \df{order} of $x$. The fact that $D$ is second countable and locally connected implies that the order of every point is countable. We will denote by $\theends(D)$, $\reg(D)$, and $\Br(D)$ the sets of endpoints, regular points, and branch points of $D$, respectively. The set $\Br(D)$ is always countable and $\theends(D)$ is a $G_\delta$ subset of $D$.

We will denote
\begin{equation}
  \label{eq:def-hat}
  \Comps{\Br(D)} = \bigsqcup_{x \in \Br(D)} \comps{x}.
\end{equation}
When $a \in \comps{x}$ will say that $a$ is a component \df{around} $x$. For a set $Z$, denote by $\Delta_Z$ the diagonal $\set{(z_1, z_2)  \in Z^2 : z_1 = z_2}$. Define the map
\begin{equation*}
  \compf \colon (\Br(D) \times D) \sminus \Delta_{\Br(D)} \to \Comps{\Br(D)}
\end{equation*}
by
\begin{equation}
  \label{eq:dfn:compf}
  \compf(x, y) = a \iff a \in \comps{x} \And  y \in a.
\end{equation}
In words, $\Phi(x, y)$ is the component around $x$ that contains $y$.

If $x, y \in D$ with $x \neq y$, we will denote by $[x, y]$ the unique arc connecting $x$ and $y$ in $D$ and we let $(x, y) = [x, y] \sminus \set{x, y}$. We define the function $K \colon D^3 \to D$ by
\begin{equation}
  \label{eq:defn:center}
  \set{\centerf(x, y, z)} = [x, y] \cap [y, z] \cap [x, z]
\end{equation}
and we say that $K(x, y, z)$ is the \df{center of $x$, $y$, and $z$}. Note that if $x$, $y$, and $z$ are branch points or distinct endpoints, then $K(x, y, z)$ is always a branch point.

An arc $[x, y]$ with $x, y \in \Br(D)$ is called \df{free} if its interior $(x, y)$ is an open set in $D$; equivalently, if $(x, y)$ does not contain any branch points. The following theorem of \Wazewski~\cite{Wazewski1923} (see also Charatonik~\cite{Charatonik1980}) is the basis for much of what follows.
\begin{theorem}[\Wazewski]
  \label{th:Wazewski}
  For every $n \in \set{3, 4, \ldots, \infty}$, there exists a unique, up to homeomorphism, dendrite $W_n$ such that:
  \begin{itemize}
  \item every branch point of $W_n$ has order $n$;
  \item $W_n$ has no free arcs.
  \end{itemize}
\end{theorem}
The second condition of the theorem is equivalent to the statement that $\Br(W_n)$ is dense in $W_n$.
It is not hard to see that $\reg(W_n)$ and $\theends(W_n)$ are also dense in $W_n$. The dendrite $W_n$ is called the \df{\Wazewski dendrite of order $n$}.

Throughout the paper, we will usually suppress the number $n \in \set{3, 4, \ldots, \infty}$ from the notation and we will write $W$ instead of $W_n$ to avoid clutter.

The homeomorphism group $\Homeo(W)$ acts transitively on each of the sets $\Br(W)$, $\reg(W)$, and $\theends(W)$ and, by density, a homeomorphism $g$ is determined uniquely by its action on each of these sets. The action $\Homeo(W) \actson \Br(W)$ defines an injective continuous homomorphism $\Homeo(W) \to \Sym(\Br(W))$, where $\Sym(\Br(W))$ is equipped with the pointwise convergence topology with $\Br(W)$ considered as a discrete set. It was shown in \cite{Duchesne2019} that this homomorphism is a topological group embedding and, consequently, the image is closed and $\Homeo(W)$ is a non-archimedean Polish group.

Next we explain the construction of kaleidoscopic permutation groups carried out in \cite{Duchesne2019a}. To lighten notation, we will denote by $X$ the set $\Br(W)$ and as in \eqref{eq:def-hat}, we write:
\begin{equation*}
  \cps = \bigsqcup_{x \in X} \comps{x}.
\end{equation*}

For the rest of the section and throughout the paper, fix a set $M$ of cardinality $n$.
A function $\kappa \colon \cps \to M$ is called a \df{coloring} if for each $x \in X$, the restriction $\kappa_x \coloneqq \kappa|_{\comps{x}}$ is a bijection.

\begin{defn}[\cite{Duchesne2019a}]
  \label{df:kaleido-coloring}
  A coloring $\kappa$ is \df{kaleidoscopic} if for all $x \neq y \in X$ and all $a \neq b \in M$, there exists $z \in X \cap (x, y)$ such that $\kappa(\Phi(z, x)) = a$ and $\kappa(\Phi(z, y)) = b$.  
\end{defn}

 It is shown in \cite{Duchesne2019a} that a kaleidoscopic coloring exists and is unique up to homeomorphism of $W$. Any coloring $\kappa$ also defines a cocycle $\alpha_\kappa \colon \Homeo(W) \times X \to \Sym(M)$ by
 \begin{equation}
   \label{eq:defn:cocycle}
  \alpha_\kappa(g, x) = \kappa_{g \cdot x} \circ g|_{\comps{x}} \circ \kappa_x^{-1},
\end{equation}
which satisfies the usual cocycle identity:
\begin{equation}
  \label{eq:cocycle-id}
  \alpha_\kappa(gg', x) = \alpha_\kappa(g, g' \cdot x) \alpha_\kappa(g', x).
\end{equation}
We call $\alpha_\kappa(g, x)$ the \df{local action of $g$ at $x$}.

Now let $\Gamma$ be a closed subgroup of $\Sym(M)$ and fix a kaleidoscopic coloring $\kappa$ of $\Comps{X}$. Define the \df{kaleidoscopic group} $\cK(\Gamma)$ by
\begin{equation}
  \label{eq:defin:KGamma}
  \cK(\Gamma) = \set{g \in \Homeo(W) : \alpha_\kappa(g, x) \in \Gamma \text{ for all } x \in X}.
\end{equation}
It follows from the uniqueness of the kaleidoscopic coloring that $\cK(\Gamma)$ does not depend on $\kappa$, up to conjugation by an element of $\Homeo(W)$.

Below we collect several facts from \cite{Duchesne2019a} that we will use.
\begin{fact}
  \label{f:kaleidoscopic-facts}
  Let $\Gamma \leq \Sym(M)$ be a closed subgroup and let $\cK(\Gamma)$ be defined as above. Then the following hold:
  \begin{enumerate}
  \item \label{i:kaleido-fact:closed} $\cK(\Gamma)$ is a closed subgroup of $\Homeo(W)$ (and therefore also of $\Sym(X)$);
  \item \label{i:kaleido-fact:oligo} The action $\cK(\Gamma) \actson X$ is oligomorphic iff the action $\Gamma \actson M$ is oligomorphic;
  \item \label{i:kaleido-fact:transitive-ends} The actions of $\cK(\Gamma)$ on $\theends(W), \reg(W), \bps$ are transitive. In particular, the action $\cK(\Gamma) \acts W$ is minimal. 
  \item \label{i:kaleido-fact:surjective-homo} For each $x \in \bps$, the map $\alpha_{\kc}(\cdot, x) \colon \kf(\Gamma)_x \to \Gamma$ is a surjective homomorphism. 
  \item \label{i:kaleido-fact:maximal-sub} The action $\cK(\Gamma) \acts \theends(W)$ is primitive, so the stabilizer $\cK(\Gamma)_\xi$ is a maximal subgroup of $\cK(\Gamma)$. 
  \end{enumerate}
\end{fact}
\begin{proof}
  We point to the statements in \cite{Duchesne2019a}.
  Items \ref{i:kaleido-fact:closed} and \ref{i:kaleido-fact:oligo} are Theorem 1.9 and Theorem 1.4, respectively.
 Item \ref{i:kaleido-fact:transitive-ends} is Proposition 5.5,
  \ref{i:kaleido-fact:surjective-homo} is Proposition 5.1, and
  \ref{i:kaleido-fact:maximal-sub} is  Corollary 5.11.
\end{proof}

The groups $\kf(\Gamma)$ are never locally compact. This can be seen in many ways; perhaps the fastest way to deduce it from what we have said so far is to apply \ref{i:kaleido-fact:transitive-ends} from the fact above: the set $\theends(W)$ is dense $G_\delta$ and its complement is dense, so it is not $F_\sigma$ and therefore cannot be a continuous image of a locally compact Polish group.


\subsection{Rooted kaleidoscopic groups}
\label{sec:rooted-kaleidoscopic}

An important role in the theory of kaleidoscopic groups is played by the stabilizer $\cK(\Gamma)_\xi$ of an endpoint $\xi \in \theends(W)$. By Fact~\ref{f:kaleidoscopic-facts}, this stabilizer does not depend on $\xi$, up to conjugation. 
In this subsection we consider a rooted variant of the kaleidoscopic construction and we describe the relation between this construction and the endpoint stabilizer for a transitive $\Gamma$.

Let $\xi \in \theends(W)$ be a fixed endpoint which we call the \df{root}.
The root allows us to define two other functions which will be useful: $\rootf \colon X \to \cps$ by
\begin{equation*}
  \rootf(x) = a \iff a \in \comps{x} \And \xi \in a
\end{equation*}
and the \df{meet}, ${\meet} \colon X^2 \to X$  by
\begin{equation*}
  x \meet y = \centerf(x, y, \xi).
\end{equation*}
Since the meet is associative, we can write $\bigwedge \bar x$ as a shorthand for $x_1 \meet \cdots \meet x_k$. 
The meet also defines a partial order $\preceq$ on $X$ by
\begin{equation}
  \label{eq:defn:prec}
  x \preceq y \iff x \meet y = x,
\end{equation}
which makes $X$ a meet semi-lattice.

We define root-kaleidoscopic colorings, the appropriate analogue of kaleidoscopic colorings (cf. Definition~\ref{df:kaleido-coloring}) for rooted groups. They have similar genericity properties but satisfy the additional condition that around each branch point the color of the component containing the root is always the same element $c$.
\begin{defn}
  \label{df:rooted-kaleidoscopic-col}
  Let $c \in M$ be a distinguished element. A \df{rooted coloring} of $W$ is a function $\pkc \colon \cps \to M$ such that for each $x$, $\pkc_x \coloneqq \pkc|_{\comps{x}}$ is a bijection and $\pkc(\rho(x)) = c$. A rooted coloring $\pkc$ is called \df{root-kaleidoscopic} if for all $x \in \bps$, $y \in \left(\xi, x\right)$ and $a \ne c$ in $M$, there is $z \in (y, x)$ with $\pkc(\compf(z, x)) = a$. 
\end{defn}

  Concerning the existence of root-kaleidoscopic colorings, it is possible to follow the approach of \cite{Duchesne2019a} and show that the generic rooted coloring is root-kaleidoscopic. However, we prefer to give a construction based on a given kaleidoscopic coloring, which will be useful later.

  Let $\kappa$ be a kaleidoscopic coloring (as in Definition~\ref{df:kaleido-coloring}).  Let $\secf \colon M \to \Sym(M)$ be a function such that $\secf(c) = 1_{\Sym(M)}$ and $\secf(a)\cdot a = c$ for all $a \in M$. Define
\begin{equation}
  \label{eq:defn:pkc}
  \pkc(a) = \secf(\kc(\rootf(x))) \cdot \kc(a), \quad \text{where } a \in \comps x.
\end{equation}

\begin{lemma}
  \label{l:root-kaleidoscopic-section}
  The function $\pkc$ defined by \eqref{eq:defn:pkc} is a root-kaleidoscopic coloring of $W$.
\end{lemma}
\begin{proof}
  That $\pkc$ is a rooted coloring is clear from the definition. To check that it is root-kaleidoscopic, let $a \in M$ and $x,y$ with $y \in (\xi, x)$ be given. By kaleidoscopicity of $\kc$, there is $z \in (y, x)$ such that $\kc(\compf(z, x)) = a$ and $\kc(\compf(z, y)) = c$. But $\rootf(z) = \compf(z, y)$, so $\secf(\kc(\rootf(z))) = \secf(c)$ is the identity.
Therefore $\pkc(\compf(z, x)) = \kc(\compf(z, x)) = a$ and $\pkc(\compf(z, y)) = \kc(\compf(z, y)) = c$.
\end{proof}

Root-kaleidoscopic colorings are also unique up to translation by an element of $\Homeo(W)_\xi$ (cf. Corollary~\ref{c:uniqueness-root-kalei}).
As in the non-rooted case, $\pkc$ defines a cocycle $\alpha_{\pkc} \colon \Homeo(W)_{\xi} \times X \to \Sym(n)_c$ using the same formula \eqref{eq:defn:cocycle}.

Now let $\Delta \acts M$ be a permutation group which fixes some $c \in M$. 
Fix a root-kaleidoscopic coloring $\pkc$. 
We define the \df{root-kaleidoscopic group} $\pkf(\Delta)$ by
\begin{equation}
  \label{eq:defin:KstarDelta}
  \pkf(\Delta) = \set{g \in \Homeo(W)_{\xi} : \alpha_{\pkc}(g, x) \in \Delta \text{ for all } x \in X}.
\end{equation}
It follows from the uniqueness of the root-kaleidoscopic coloring that $\pkf(\Delta)$ does not depend on $\pkc$, up to conjugation by an element of $\Homeo(W)_\xi$. When $\Gamma$ is transitive, the root-kaleidoscopic group of a point stabilizer is equal to the stabilizer of $\xi$ in $\kf(\Gamma)$, as the next proposition shows.

\begin{prop}
  \label{prop:KstarGamma_c-is-stabilizer}
  Let $\Gamma \acts M$ be a transitive permutation group, let $c \in M$, let $\kappa$ be a kaleidoscopic coloring of $\Comps{X}$, and let $\pkc$ be the root-kaleidoscopic coloring given by \eqref{eq:defn:pkc}. Then
  \begin{equation*}
    \pkf(\Gamma_c) = \kf(\Gamma)_\xi,
  \end{equation*}
  where the groups $\pkf(\Gamma_c)$ and $\kf(\Gamma)$ are constructed using the colorings $\pkc$ and $\kappa$, respectively.
\end{prop}
\begin{proof}
  By transitivity of $\Gamma$, there is $\sigma \colon M \to \Gamma$ such that $\sigma(c) = 1_\Gamma$ and $\sigma(a) \cdot a = c$ for all $a \in M$. Suppose that $\kc$ is a kaleidoscopic coloring and $\pkc$ is defined by \eqref{eq:defn:pkc} from $\kc$ and $\sigma$.
  It follows from the uniqueness of root-kaleidoscopic colorings that $\pkc$ gives rise to $\pkf(\Gamma_c)$ by \eqref{eq:defin:KstarDelta}.
  From the definition \eqref{eq:defn:pkc} it is clear that for all $g \in \Homeo(W)$ and $x \in X$,
  \begin{equation*}
  \alpha_{\pkc}(g, x) \in \Gamma_c \text{ if and only if } \alpha_{\kc}(g, x) \in \Gamma.
\end{equation*}
Now, looking at the definitions \eqref{eq:defin:KGamma} and \eqref{eq:defin:KstarDelta}, we obtain the desired conclusion.
\end{proof}


\section{Structures for kaleidoscopic groups and endpoint stabilizers}
\label{sec:kaleidoscopic_structures}

In this section and the next lies the combinatorial core of this paper. 
The goal is to prove the following statements, from which we will derive our description of the UMFs of kaleidoscopic groups.
If $\Gamma \acts M$ is a transitive permutation group and $\Delta \le \Gamma_c$ is a closed subgroup, for some $c \in M$, then:
\begin{itemize}
  \item  if $\Delta$ is presyndetic in $\Gamma_c$, then $\pkf(\Delta)$ is presyndetic in $\pkf(\Gamma_c) = \kf(\Gamma)_{\xi}$ (\Cref{pr:H-presyndetic}).
  \item  if $\Delta$ is oligomorphic then $\pkf(\Delta)$ is oligomorphic (\Cref{c:KstarDeltaOligo}).
  \item  if $\Delta$ is extremely amenable, so is $\pkf(\Delta)$ (\Cref{cor:KstarDeltaExtAmenable}).  
\end{itemize}

The first step is to reduce them to combinatorial problems, by representing $\kf(\Gamma)$, $\pkf(\Gamma_c)$, and $\pkf(\Delta)$ as automorphism groups of appropriate ultrahomogeneous structures.


\subsection{Structures for the homeomorphism groups and endpoint stabilizers}
\label{sec:struct-homeo-groups}

First we explain how to view the full homeomorphism group of the \Wazewski dendrite $W_n$ as the automorphism group of an appropriate ultrahomogeneous structure. A similar approach was already used by Kwiatkowska~\cite{Kwiatkowska2018}; our structure is somewhat different as we use a signature with function symbols rather than a relational one, and it is better adapted to our treatment of kaleidoscopic groups.

For the rest of this section, we fix $n \in \set{3, 4, \ldots, \infty}$ and as before, we will write $\waz$ instead of $W_n$. We consider a signature $\cL_\bX$ with one ternary function symbol $\centerf$ and we define a structure $\bX$ with universe $X \coloneqq \Br(W)$ by interpreting $\centerf$ as the (restriction of) the center function given by \eqref{eq:defn:center}.
It is a particular case of \Cref{p:KalM-is-homogeneous} that $\bX$ is ultrahomogeneous. 

If $A \sub X$ is a finite set, the substructure of $\bX$ generated by $A$ is the smallest center-closed subset of $X$ containing $A$; it can be visualized as a combinatorial tree whose geometric realization embeds as a subdendrite of $W$. A finitely generated substructure of $\bX$ is always finite. We denote $\cF_\bX = \Age(\bX)$.

The function $K$ allows us to define (via a quantifier-free formula) the \df{betweenness relation} $B$ by:
\begin{equation*}
  B(x, y, z) \iff K(x, y, z) = y.
\end{equation*}
Conversely, $K$ can be defined from $B$ by
\begin{equation*}
  K(x_1, x_2, x_3) = y \iff B(x_1, y, x_2) \And B(x_1, y, x_3) \And B(x_2, y, x_3).
\end{equation*}
Thus for checking whether a certain map between $\cF_\bX$-structures is an embedding, it suffices to see that it preserves $B$. This observation allows us to see that our structure is enough to recover the dendrite, at least as far as automorphisms are concerned.
\begin{prop}
    \label{prop:aut-group-of-the-dendrite}
    $\Aut(\wazs) = \Homeo(\waz)$.
\end{prop}
\begin{proof}
  By \cite{Duchesne2019}*{Proposition 2.4} and \cite{Duchesne2019a}*{Proposition 2.4}, the group of homeomorphisms of $\waz$ is equal to the group of permutations of $\bps$ which preserve $\betr$. 
\end{proof}

If $\bT \in \cF_{\bX}$ and $x \in T$, we can define, as before, $\comps{x}$ as the set of connected components of $T \sminus \set{x}$ (where we view $T$ as a geometric tree) and $\Comps{T} = \bigsqcup_{x \in T} \comps{x}$. This construction can be captured as an imaginary sort for $\cF_\bX$. Indeed, consider the formulas
\begin{equation*}
  \theta(x, y) \equiv \big(x \neq y\big) \And \eps(x_1, y_1, x_2, y_2) = \big(x_1 = x_2 \land \neg B(y_1, x_1, y_2) \big).
\end{equation*}
It is clear that $\eps$ is precisely the equivalence relation ``$x_1 = x_2$ and $y_1$ and $y_2$ are in the same component of $T \sminus \set{x_1}$''. For a general $\cF_\bX$-structure $\bY$, we will denote by $\Comps{Y}$ the elements of this sort.

Now we describe a structure whose automorphism group is the stabilizer $\Homeo(W)_{\xi}$ for an endpoint $\xi \in \theends(W)$, which we call the root, as in \Cref{sec:rooted-kaleidoscopic}.
Let $\cL^*_\bX$ be the signature $\cL_\bX \sqcup \set{\bxi}$, where $\bxi$ is a constant symbol. Let $\bX^*$ be the $\cL^*_\bX$-expansion of $\bX$ obtained by taking $X^* = X \sqcup \set{\xi}$, interpreting $\bxi$ as $\xi$, and interpreting the function $\centerf$ as usual by \eqref{eq:defn:center}.
Notice that this respects our definition of expansion since $\xi \in \dcl(\emptyset)$ and $X$ is $\Aut(\bX^*)$-invariant. 
Recall that the root allows us to define $\rootf \colon X \to \cps$, the meet ${\meet} \colon (X^*)^2 \to X^*$  and the partial order $\prec$ on $X^*$.
All of their definitions are by quantifier-free formulas, so they are also valid for the age of $\bX^*$.
    
Similarly as before, we have the following.
\begin{prop}
  \label{p:autom-Xstar}
  $\Aut(\bX^*) = \Homeo(W)_\xi$.
\end{prop}
\begin{proof}
  The $\supseteq$ inclusion follows from Proposition~\ref{prop:aut-group-of-the-dendrite} and the definition of $\bX^*$. For the other, suppose that $g \in \Aut(\bX^*)$, viewed as a homeomorphism of $W$ via the identification $\Aut(\bX^*) \sub \Aut(\bX) = \Homeo(W)$, moves $\xi$ to another endpoint $\xi'$. Let $x \neq y \in X$ be such that $\xi, x, y, \xi'$ lie on an arc, in this order. Then, on the one hand, we have that $B(x, y, \xi')$ and as $g$ is a homeomorphism, we obtain that $B(g^{-1} \cdot x, g^{-1} \cdot y, \xi)$. On the other, $g$ is an automorphism of $\bX^*$, so it preserves the binary relation $B(\cdot, \cdot, \xi)$ and the fact that $B(y, x, \xi)$ implies that $B(g^{-1} \cdot y, g^{-1} \cdot x, \xi)$, contradiction.
\end{proof}


\subsection{Structures for kaleidoscopic groups}
\label{sec:struct-kale-groups}

Next we turn to kaleidoscopic groups. Let $\cL_0$ be a relational signature, let $\bM$ be an ultrahomogeneous $\cL_0$-structure with $|M| = n$ and let $\Gamma = \Aut(\bM)$. (As we explained earlier, this entails no loss of generality: any closed subgroup $\Gamma \leq \Sym(M)$ can be represented in this fashion.)

Fix $\kc\colon \cps \to M$, a kaleidoscopic coloring of $\waz$. 
We use $\kappa$ to define an expansion $\kf(\bM)$ of $\wazs$. 

\begin{defn}
  \label{def:kal-strcture}
  Let $\kf(\bM)$ be the relational expansion of $\wazs$ in the signature $\cL_\bX \sqcup \cL_0$, where the relations in $\cL_0$ are interpreted on the sort $\cps$, as follows. 
  For a $k$-ary relation $R \in \cL_0$ and $\bar a = (a_0, \ldots, a_{k-1}) \in \cps^k$, we define
  \begin{equation*}
    R^{\cK(\bM)}(\bar a) \iff \exists x \in X \ a_0, \ldots, a_{k-1} \in \comps{x} \And R^\bM(\kappa_x(\bar a)).
  \end{equation*}
\end{defn}

In other words, for all $x \in X$, $\kappa_x \colon \comps{x} \to M$ is an $\cL_0$-isomorphism.

It follows easily from the results of \cite{Duchesne2019a} that this structure captures the kaleidoscopic construction, as we see in the following proposition.
\begin{prop}
  \label{p:KalM-is-homogeneous}
  The groups $\cK(\Gamma)$ and $\Aut(\cK(\bM))$ coincide. Moreover, the structure $\kf(\bM)$ is ultrahomogeneous and does not depend on the choice of kaleidoscopic coloring, up to isomorphism. 
\end{prop}
\begin{proof}
  The first assertion follows from the definitions.
  
  Next we prove ultrahomogeneity. 
  Let $\bS$ and $\bT$ be two finite substructures of $\kf(\bM)$ and let $f \colon \bS \to \bT$ be an isomorphism. 
  For every $\bs \in S$,
  \begin{equation*}
    \kc_{f(\bs)} \circ f \circ \kc_\bs^{-1} \colon \kc_\bs[\comps{\bs}] \to \kc_{f(\bs)}[\Comps{f(\bs)}]
  \end{equation*}
  is an isomorphism of finite substructures of $\bM$. As $\bM$ is ultrahomogeneous, this isomorphism extends to some $\gamma_\bs \in \Aut(\bM)$. Now it follows from \cite{Duchesne2019a}*{Lemma 5.8} that there is $g \in \kf(\Gamma) = \Aut(\kf(\bM))$ which extends $f$.

  Finally, let $\kc' \colon \cps \to M$ be some other kaleidoscopic coloring. 
  By \cite{Duchesne2019a}*{Theorem 3.4}, there is $h \in \Homeo(\waz)$ such that $\kc = \kc' \circ h$ and it is easy to see that $h$ defines an isomorphism between $\cK_\kappa(\bM)$ and $\cK_{\kappa'}(\bM)$.
\end{proof}


\subsection{Structures for rooted kaleidoscopic groups}
\label{sec:rooted-kaleidoscopic-structures}

Let now $\cL_0$ be a relational signature as before and let $\bN$ be an ultrahomogeneous structure in the signature $\cL_0 \sqcup \set{\cc}$, where $\cc$ is a constant symbol. As usual, let $N$ denote the universe of $\bN$ and suppose that $|N| = n$. We denote by $c \in N$ the interpretation of the constant symbol $\cc$. 

We also fix some root-kaleidoscopic coloring $\pkc$ of $W$, with $\pkc(\rootf(x)) = c$ for all $x \in \bps$. 
We use it to define the expansion $\pkf(\bN)$ of the structure $\bX^*$.

\begin{defn}
  Let $\pkf(\bN)$ be the relational expansion of $\bX^*$ to the signature $\cL^*_\bX \sqcup \cL_0$ defined as follows:
  for each $\rR \in \cL_0$ of arity $k$, we interpret $\rR$ as a relation on $\cps$, which for $\bar a = (a_0, \ldots, a_{k-1}) \in \cps^k$, is given by
  \begin{equation*}
    \rR^{\pkf(\bN)}(\bar a) \iff \exists x \in X \ a_0, \ldots, a_{k-1} \in \comps{x} \And \rR^\bN(\pkc(\bar a)).
  \end{equation*}
\end{defn}  

It will be convenient to establish some combinatorial notation and terminology. For us, a \df{combinatorial tree} is just an element of the age of $\bX$. Two vertices of a tree are \df{adjacent} if there is no other vertex between them. 
Similarly, we call a \df{rooted tree} an element of the age of $\bX^*$. Note that every combinatorial tree as well as every rooted tree is finite. The \df{root} of a rooted tree is the interpretation of $\bxi$. A \df{leaf} of a rooted tree is a $\prec$-maximal element. The \df{height} of an element $t$ of a rooted tree $\bT$ is the number $\abs{\set{s \in T : s \prec t \And s \neq t}}$. The \df{height} of the tree is the maximum of the heights of its elements. If $t \in T$, we denote by $\Succ(t)$ the set of $\prec$-immediate successors of $t$. Notice that the immediate successors are adjacent to $t$.
If $A \sub T$, we denote by $\bigwedge A$ the $\prec$-greatest lower bound of $A$. Every rooted tree $\bT$ with more than one vertex also has another distinguished vertex apart from the root, namely $\bigwedge (T \sminus \bxi)$, the unique immediate successor of $\bxi$. As we will use this vertex frequently, we denote it by $r_{\bT}$. Note, however, that $r_{\bT}$ is not preserved by embeddings.
For $t \in T$, let
\begin{equation*}
\subtree{\bT}{t} \coloneqq \set{s \in T : t \preceq s} \cup \set{\xi}.
\end{equation*}
So $t = r_{\subtree{\bT}{t}}$.

Our next step is to describe a cofinal family in the age of $\pkf(\bN)$ that is easier to work with.
\begin{defn}
  \label{df:nice-A-h}
  Let $\bA$ be a finite structure in the signature $\cL_0 \cup \set{\cc}$ and let $h$ be a natural number. We define $\nice{\bA}{h}$ to be the finite $\cL^*_\bX \sqcup \cL_0$-structure such that: 
\begin{itemize}
\item the reduct of $\nice{\bA}{h}$ to $\cL^*_\bX$ is a rooted tree whose every leaf has height $h+1$ and such that every non-leaf different from $\bxi$ has exactly $|A|-1$ immediate successors;

\item for any $\bs \in A[h]$ which is not a leaf and not equal to $\bxi$, the $\cL_0 \sqcup \set{\cc}$-structure on $\comps{s}$, given by interpreting $\cc$ as $\rootf(\bs)$, is isomorphic to $\bA$. We denote by $\lambda_s$ this isomorphism.
\end{itemize}
\end{defn}

\begin{figure}
  \label{fig:Ah-structures}
  \centering{
    \includeinkscape[scale=0.7]{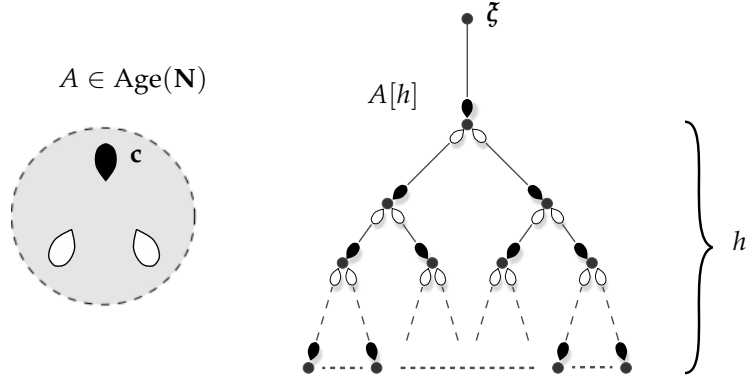}
    }
    \caption{A graphical representation of $A[h]$.}
\end{figure}

\begin{prop}
  \label{prop:cofinal-in-point-kal}
  The family of structures
  \begin{equation*}
    \set{\nice{\bA}{h} : \bA \in \Age(\bN), h = 0, 1, \ldots}
  \end{equation*}
  is cofinal in $\Age(\pkf(\bN))$.
\end{prop}
\begin{proof}
  First we show that for any $\bA \in \Age(\bN)$ and $h \in \N$, there is a substructure of $\pkf(\bN)$ isomorphic to $\nice{\bA}{h}$. 
  We do so by fixing $\bA \sub \bN$ and arguing by induction on $h$. 
  For $h = 0$, choose any $x \in \bps$. Then the generated substructure $\gen{x} = \set{x, \xi}$ of $\bX^*$ is isomorphic to $\bA[0]$.
  Now suppose we have found $\bS \sub \pkf(\bN)$ which is isomorphic to $\nice{\bA}{h-1}$, for some $h > 0$. 
  For each leaf $x \in S$, and each $a \in A \setminus \set{c}$,  pick a point $x_a \in (\pkc_x)^{-1}(a)$. 
  Let $\bS'$ be the substructure of $\cK^*(\bN)$ consisting of $\bS$ and all of these new points $x_a$. Notice that this set is closed under $\centerf$. It is also easy to see that $\bS'$ is isomorphic to $\nice{\bA}{h}$.   

  Conversely, let $\bS$ be a finite substructure of $\pkf(\bN)$. Let $A = \set{\pkc(a) : a \in \Comps{S}}$ and let $\bA$ be the substructure of $\bN$ with universe $A$. We also let $h$ be the height of $\bS$ and we define an embedding $f$ of $\bS$ into $\nice{\bA}{h}$ by induction as follows. First we let $f(\xi) = \xi$ and $f(r_\bS) = r_{\nice{\bA}{h}}$. Then, supposing that for some non-leaf $s \in S$, $f(s)$ has been defined, we let
  \begin{equation*}
    f|_{\Succ(s)} = \lambda_{f(s)}^{-1} \circ \kappa_s^*.
  \end{equation*}
  It is easy to check that $f$ is an embedding.
\end{proof}

Now let $\Delta \acts N$ be a permutation group which fixes some $c \in N$. 
Let $\bN$ be an ultrahomogeneous structure in the signature $\cL_0 \cup \set{\cc}$, where $\cc$ is a constant symbol, interpreted as $c$, such that $\Aut(\bN) = \Delta$.

\begin{theorem}
  \label{thm:point-kal-is-homogeneous}
  The groups $\pkf(\Delta)$ and $\Aut(\pkf(\bN))$ coincide.
  Moreover, $\pkf(\bN)$ is ultrahomogeneous and does not depend on the choice of $\pkc$, up to isomorphism.
\end{theorem}
\begin{proof}
  The first assertion follows from the definitions. 

  To verify ultrahomogeneity, we apply the standard back-and-forth argument. Given an isomorphism $f \colon \bS \to \bT$ between finite substructures of $\pkf(\bN)$ and a point $x \in \pkf(\bN) \sminus S$, it suffices to extend $f$ to an embedding $f' \colon \gen{S, x} \to \pkf(\bN)$.
      
  Denote $\bS' = \gen{S, x}$ and let $y$ be the immediate $\prec$-predecessor of $x$ in $\bS'$. 
  There are three cases depending on the position of $y$. 
    
  \textbf{Case 1.} $y \in S$. Let $a = \compf(y, x)$.
  The isomorphism $f$ induces an isomorphism $\tilde f \colon \pkc_y[\Comps{S}] \to \pkc_{f(y)}[\Comps{T}]$ between finite substructures of $\bN$. Since $\bN$ is ultrahomogeneous, there is $\tilde g \in \Aut(\bN)$ that extends $\tilde f$. 
  Let $b = (\pkc_{f(y)})^{-1}(\tilde g(\pkc(a)))$. Finally, choose any $x' \in b \sminus T$ and let $f' = f \cup \set{(x, x')}$. 

  \textbf{Case 2.} $y \in (\xi, r_\bS)$. Let $a = \compf(y, r_\bS)$ and notice that $\pkc(a) \ne c$. By root-kaleidoscopicity of $\pkc$, there is $y' \in (\xi, f(r_\bS))$ such that $\pkc(\compf(y', f(r_\bS))) = \pkc(a)$. We let $f'= f \cup \set{(y, y')}$. If $x = y$, we are done; if not, we have reduced to Case~1.

  \textbf{Case 3.} There are $z_1, z_2 \in S$ adjacent in $\bS$, with $z_1 \in (\xi, z_2)$ and such that $y \in (z_1, z_2)$. Let $a = \compf(y, z_2)$ and note that $\pkc(a) \ne c$. By root-kaleidoscopicity of $\pkc$, there is $y' \in (f(z_1), f(z_2))$ such that $\pkc(\compf(y', f(z_2))) = \pkc(a)$. We let $f'= f \cup \set{(y, y')}$. If $x = y$, we are done; otherwise, we have again reduced to Case~1. This concludes the proof of ultrahomogeneity.

  If ${\pkc}'$ is another root-kaleidoscopic coloring, then $\pkf_{\pkc}(\bN)$ and $\pkf_{{\pkc}'}(\bN)$ are two ultrahomogeneous structures with the same age (by Proposition~\ref{prop:cofinal-in-point-kal}), so they are isomorphic by \Fraisse's theorem.
\end{proof}

By considering a structure $\bN$ with trivial automorphism group, we obtain the following.
\begin{cor}
  \label{c:uniqueness-root-kalei}
  If $\kappa^*_1$ and $\kappa^*_2$ are two root-kaleidoscopic colorings, then there exists $g \in \Homeo(W)_\xi$ such that $\kappa^*_2 = \kappa^*_1 \circ g$.
\end{cor}
\begin{proof}
  If $g$ is an isomorphism between $\pkf_{\pkc_1}(\bN)$ and $\pkf_{\pkc_2}(\bN)$, then $g$ preserves betweenness on $X$ and $\xi$, so $g \in \Homeo(\waz)_\xi$. Since $\bN$ has trivial automorphism group, it follows that $\kappa^*_2 = \kappa^*_1 \circ g$.
\end{proof}

\begin{cor}
  \label{c:KstarDeltaOligo}
  Let $\Delta \acts N$ be a permutation group which fixes some $c \in N$.
  Then $\pkf(\Delta) \acts \bps$ is oligomorphic if and only if $\Delta \acts N$ is.
\end{cor}
\begin{proof}
  $(\Leftarrow)$ 
  Let $\bN$ be an ultrahomogeneous structure such that $\Aut(\bN) = \Delta$, as above. The hypothesis means that we can choose a signature for $\bN$ that contains, apart from $\cc$, finitely many relation symbols in each arity. Theorem~\ref{thm:point-kal-is-homogeneous} implies that the conclusion is equivalent to: for every $k$, there are only finitely many $k$-generated structures in $\Age(\pkf(\bN))$, up to isomorphism. This follows from the fact that a $k$-generated substructure of $\bX^*$ has at most $2k$ elements and that a finite substructure $\bT \sub \pkf(\bN)$ is determined by its reduct to the signature of $\bX^*$ and the isomorphism types of $\comps{t}$ for $t \in T$, as elements of $\Age(\bN)$.
  
  $(\Rightarrow)$ Suppose there is $k > 0$ such that there are infinitely many pairwise non-isomorphic $\bA_1, \bA_2, \ldots$ elements of $\Age(\bN)$ with $k+1$ elements (including $\cc$). Then $\bA_1[1], \bA_2[1], \ldots$ are $k$-generated non-isomorphic elements of $\Age(\pkf(\bN))$.
\end{proof}

Now fix $x \in \bps$ and a component $a \in \comps{x}$ which does not contain the root $\xi$.
Then $a \cup \set{x}$ is a dendrite homeomorphic to $W$ (by Theorem~\ref{th:Wazewski}) and $\pkc$ restricted to it is again a root-kaleidoscopic coloring, this time with root $x$.
By Corollary~\ref{c:uniqueness-root-kalei}, there exists a homeomorphism $h_a \colon \waz \to a \cup \set{x}$ mapping $\xi$ to $x$ and such that $\pkc = \pkc \circ h_a$. Using these homeomorphisms, one can run the argument in \cite{Duchesne2019a}*{Proposition 5.1} to obtain the following.
\begin{cor}
  \label{cor:homo-from-open-subgroup-rooted}
  Let $\Delta \acts n$ be a permutation group which fixes some $c \in n$ and let $x \in \bps$.
  Then the homomorphism $\alpha_{\pkc}(\cdot, x) \colon \pkf(\Delta)_x \to \Delta$ is \df{split-surjective}, i.e., there exists a continuous homomorphism $\phi \colon \Delta \to \pkf(\Delta)_x$ with a closed image such that $c_\pkc(\cdot, x) \circ \phi = \id_\Delta$.
\end{cor}


\subsection{The endpoint stabilizer and root-kaleidoscopic groups}
\label{sec:endpoint-stabilizer}

Let $\cL_0$ be a relational signature, let $\bM$ be an ultrahomogeneous $\cL_0$-structure with universe $M$ and let $\Gamma = \Aut(\bM)$. Assume that the action $\Gamma \actson M$ is transitive. Let $c \in M$ and let $\Delta$ be a closed subgroup of $\Gamma_c$. In this subsection, we explain how to represent $\pkf(\Delta)$ as the automorphism group of $\pkf(\bN)$ for an appropriate expansion $\bN$ of the structure $\bM$ satisfying $\Aut(\bN) = \Delta$.

Let $\kappa$ be some fixed kaleidoscopic coloring of $W$. We have seen in Proposition~\ref{p:KalM-is-homogeneous} that $\kf(\Gamma) = \Aut(\kf(\bM))$. Let $\bM_\cc$ be the expansion of $\bM$ to the signature $\cL_0 \sqcup \set{\cc}$ obtained by interpreting the symbol $\cc$ as $c$. Then $\Aut(\bM_\cc) = \Gamma_c$. Let $\bN$ be an expansion of $\bM_\cc$ to a signature $\cL_0 \sqcup \cL_1 \sqcup \set{\cc}$ such that $\cL_1$ is relational, $\bN$ is ultrahomogeneous, and $\Aut(\bN) = \Delta$. 

We have that $\pkf(\Gamma_c) = \kf(\Gamma)_\xi$ by \Cref{prop:KstarGamma_c-is-stabilizer}, and therefore that $\pkf(\Delta) \le \pkf(\Gamma_c) \le \kf(\Gamma)$. We see how this is reflected at the level of structures and expansions. 

\begin{prop}
  \label{p:point-stabilizer}
  Let $\Gamma \actson M$ be transitive and let $\bM$, $\bM_\cc$ and $\bN$ be defined as above. Then the following hold:
  \begin{enumerate}
  \item \label{i:ps:1} $\pkf(\bN)$ is a relational expansion of $\pkf(\bM_\cc)$ and $\pkf(\bM_\cc)$ is an expansion of $\kf(\bM)$;
  \item \label{i:ps:2} $\pkf(\Delta) =  \Aut(\pkf(\bN))$ and $\pkf(\Gamma_c)  = \Aut(\pkf(\bM_\cc))= \kf(\Gamma)_\xi$.  
  \end{enumerate}
\end{prop}
\begin{proof}
  \ref{i:ps:1} It is clear that $\pkf(\bN)$ is a relational expansion of $\pkf(\bM_\cc)$. The only symbol of the signature of $\pkf(\bM_\cc)$ which is not in the signature of $\kf(\bM)$ is the constant symbol $\bxi$ and the respective universes of $\kf(\bM)$ and $\pkf(\bM_\cc)$ are $X$ and $X \cup \set{\xi}$. Thus to check the claim, it suffices to prove that the identity on $X$ is an embedding, that is, that the interpretations of the other symbols agree on $X$ and $\cps$. We have already checked that $\bX^*$ is an expansion of $\bX$. For a relation symbol $R$ in $\cL_0$ and $\bar a \in \comps{x}^k$, we have
  \begin{equation*}
    R^{\pkf(\bM_\cc)}(\bar a) \iff R^{\bM_\cc}(\kappa^*(\bar a)) \iff R^\bM(\kappa(\bar a)) \iff R^{\kf(\bM)}(\bar a).
  \end{equation*}
  The middle equivalence follows from the definition \eqref{eq:defn:pkc} of $\pkc$ and the fact that we can use $\sigma$ which take values in $\Gamma = \Aut(\bM)$.

  \ref{i:ps:2} This is \Cref{prop:KstarGamma_c-is-stabilizer} and \Cref{thm:point-kal-is-homogeneous} applied to $\Aut(\bM_\cc) = \Gamma_c$.
\end{proof}

\begin{prop}
  \label{pr:H-presyndetic}
  If $\Gamma \acts M$ is a transitive permutation group and $\Delta \le \Gamma_c$ is a closed presyndetic subgroup, for some $c \in M$, then $\pkf(\Delta)$ is presyndetic in $\kf(\Gamma)_\xi$ and in $\kf(\Gamma)$.
\end{prop}
\begin{proof}
  Let $\bM$, $\bM_\cc$ and $\bN$ be defined as above.
  We assume that $\bN$ is a minimal expansion of $\bM_\cc$ and prove that $\pkf(\bN)$ is a minimal expansion of $\pkf(\bM_\cc)$.
  Fix $\bB' \in \Age(\bN)$ and let $\bA \in \Age(\bM_\cc)$ be a witness of minimality for $\bB'$.
  We show that $\nice{\bA}{h}$ is a witness of minimality for $\nice{\bB'}{h}$, for each $h \in \N$.
  By \Cref{prop:cofinal-in-point-kal} and Proposition~\ref{p:presyndetic-criterion}, this will be enough. 
    
  So fix an expansion $\bS$ of $\nice{\bA}{h}$.
  Then $\bS \restr{\cL_\bX} = \nice{\bA}{h}\restr{\cL_\bX}$.
  For each $\bs \in \bS$, which is not a leaf nor $\bxi$, the structure $\pkc[\comps{\bs}]$ is an expansion of $\bA$, and thus there is an embedding $g_{\bs} \colon \bB' \to \pkc[\comps{\bs}]$. 
  Therefore there is a level-preserving embedding $f$ of $\nice{\bB'}{h}$ into $\bS$ which is given by mapping the $r_{\nice{\bB}{h}}$ to $r_{\bS}$ and, if $f$ is defined for $p$ and $s$ is an immediate successor of $p$, letting $f(s)$ be the immediate successor of $f(p)$ which belongs to the component $(\pkc_{f(\bp)})^{-1} \circ g_{f(\bp)}\circ \pkc_p(\compf(p, s))$.

  Finally, note that as the action $\kf(\Gamma) \actson W$ is minimal and the orbit $\kf(\Gamma) \cdot \xi$ is $G_\delta$, it follows from Proposition~\ref{p:SG-comeager} and the fact that highly proximal extensions preserve minimality that the action $\kf(\Gamma) \actson \Sam(\kf(\Gamma)/\kf(\Gamma)_\xi)$ is minimal. Thus, by \cite{Zucker2021}*{Proposition~6.6}, $\kf(\Gamma)_\xi$ is presyndetic in $\kf(\Gamma)$ and by Proposition~\ref{p:presyndetic-transitive} and what we already proved, $\pkf(\Delta)$ is also presyndetic in $\kf(\Gamma)$.
\end{proof}


\subsection{Rendering structures full}
\label{sec:rend-struct-full}

In our representation theorems of flows as type spaces, it will be important that we deal with full structures. It is unfortunately not true in general that $\kf(\bM)$ is full, even if we start with a full structure $\bM$, unless $\Aut(\bM)$ is oligomorphic (or, equivalently, $\bM$ is $\aleph_0$-categorical in a countable language). 
So our remedy is to render them qf-full by expanding the language to include relations for all $\Aut(\kf(\bM))$-invariant sets of tuples. We denote this full expansion by $\fullkf(\bM)$ and its signature by $\cl{\cL}$ and note that by Theorem~\ref{th:aleph0-cat}, it eliminates quantifiers (and, of course, its automorphism group is equal to $\Aut(\kf(\bM))$). 
If $\Aut(\bM)$ is transitive and $c \in M$, we define $\fullpkf(\bM_\cc)$ as the relational expansion of $\pkf(\bM_\cc)$ to the signature $\cl{\cL}_{\bxi} \coloneqq \cl{\cL} \sqcup \set{\bxi}$. Note that this expansion does not change the automorphism group of $\pkf(\bM_\cc)$ and that for any quantifier-free $\cl{\cL}$-formula $\phi(\bar v)$ and $\bar x \in X^{\bar v}$, we have that $\fullpkf(\bM_\cc) \models \phi(\bar x) \iff \fullkf(\bM) \models \phi(\bar x)$.

We will call a formula $\phi(\bar v)$ \df{isolating} if it isolates a type, i.e., $\phi(\fullkf(\bM))$ is a single $\Aut(\kf(\bM))$-orbit. Then every $\cl{\cL}$-formula is equivalent in $\fullkf(\bM)$ to a (possibly infinite) disjunction of isolating formulas.

\begin{lemma}
  \label{l:fullkf-fullkpf}
  Suppose that $\Aut(\bM) \actson M$ is transitive, $\bM$ is ultrahomogeneous, and $c \in M$, and define $\kf(\bM)$ and $\pkf(\bM_\cc)$ as before. Then the structure $\fullpkf(\bM_\cc)$ is full and, moreover, every formula is equivalent in $\fullpkf(\bM_\cc)$ to an existential one.
\end{lemma}
\begin{proof}
  Let $G = \Aut(\kf(\bM))$ and recall that then $G_\xi = \Aut(\pkf(\bM_\cc))$. 
  Let $D \sub (X\cup \set{\xi})^k$ be a $G_\xi$-invariant set which we want to prove is definable.
  Up to taking a conjunction with a formula stating which of the $k$ variables are equal to $\bxi$, we can assume that tuples in $D$ do not contain $\xi$, i.e., $D \sub X^k$. 
  Let us write $D = \bigsqcup_{\alpha} D_{\alpha}$, where each $D_{\alpha}$ is a $G_\xi$-orbit. 

  By ultrahomogeneity, for each $\alpha$, the tuples in $D_{\alpha}$ are exactly those of a given type, so there is some $\cl{\cL}_{\bxi}$-formula $\theta_{\alpha}(\bar v)$ which isolates it.  
  We may assume that $\compf(\xi, v_i)$ does not occur in $\theta_{\alpha}$, for any $1 \le i \le k$, since no atomic formulas mentioning the unique component around $\xi$ hold in $\fullpkf(\bM_\cc)$.

  Let $\phi_\alpha(u, \bar v)$ be the $\cl{\cL}$-formula obtained by substituting each occurrence of $\bxi$ in $\theta_\alpha$ by $u$.

  \begin{claim*}
    For any $\bar x \in X^k$, $\pkf(\bM_\cc) \models \theta_{\alpha}(\bar x)$ if and only if for some (equivalently, for all) $z \in \rho(\bigwedge \bar x)$, it holds that $\kf(\bM) \models \phi_\alpha(z, \bar x)$.
  \end{claim*}
  \begin{proof}[Proof of Claim]
    Since $\compf(\bxi, v_i)$ does not occur in $\theta_\alpha$, neither does $\compf(u, v_i)$ in $\phi_\alpha$. 
    Therefore it is only a matter of noting that for any $z \in \rho(\bigwedge \bar x)$ and $1 \le i, j \le k$, we have $\centerf(z, x_i, x_j) = \centerf(\xi, x_i, x_j)$ and $\compf(x_i, z) = \compf(x_i, \xi)$. 
  \end{proof}
  Using the qf-fullness of $\fullkf(\bM)$, let $\phi(u, \bar v)$ be the quantifier free $\cl{\cL}$-formula equivalent to the possibly infinite disjunction  $\bigvee_\alpha \phi_\alpha(u, \bar v)$ in $\fullkf(\bM)$. 
  Then we claim that the formula
  \begin{equation}
    \label{eq:equiv-Lstar}
    \psi(\bar v) \coloneqq \exists u \in \rho(\bigwedge \bar v) \ \phi(u, \bar v)
  \end{equation}
  defines $D$.
  If $\bar x \in D$, then there is $\alpha$ such that $\bar x \in D_\alpha$, so $\pkf(\bM_\cc) \models \theta_\alpha(\bar x)$. 
  The Claim tells us that for any $z \in \rho(\bigwedge \bar v)$, $\fullkf(\bM) \models \phi_\alpha(z, \bar x)$. 
  As $\phi_\alpha$ is quantifier-free, this implies that $\fullpkf(\bM_\cc) \models \phi_\alpha(z, \bar x)$, so $\fullpkf(\bM_\cc) \models \psi(\bar x)$.

  Conversely, suppose that $\fullpkf(\bM_\cc) \models \psi(\bar x)$. 
  It follows that there are $z \in \rho(\bigwedge \bar x)$ and $\alpha$ such that $\fullpkf(\bM_\cc) \models \phi_\alpha(z, \bar x)$.
  Since $\phi_\alpha$ doesn't mention $\bxi$, $\fullkf(\bM) \models \phi_\alpha(z, \bar x)$, which, by the Claim, shows that $\fullpkf(\bM_\cc) \models \theta_\alpha(\bar x)$, so $\bar x \in D_\alpha \sub D$.
\end{proof}


\section{A Ramsey interlude}
\label{sec:ramsey-interlude}

The goal of this section is to prove a transfer Ramsey theorem that will be our main combinatorial tool for calculating universal minimal flows of kaleidoscopic groups. First we recall the definition of the Ramsey property.

If $\bA$ and $\bB$ are two structures in the same signature, we denote by $\binom{\bB}{\bA}$ the set of embeddings of $\bA$ into $\bB$. For three structures $\bA, \bB, \bC$ and $k \in \N$, we denote by $\bC \to (\bB)_\bA^k$ the statement that $\binom{\bC}{\bA} \neq \emptyset$ and for any coloring $\gamma \colon \binom{\bC}{\bA} \to k$, there exists $f \in \binom{\bC}{\bB}$ such that $\gamma$ is constant on the set $f \circ \binom{\bB}{\bA}$.

Let $\cF$ be an age. We will say that $\bA \in \cF$ is a \df{Ramsey structure} for $\cF$ if for every $\bB \in \cF$, there exists $\bC \in \cF$ such that $\bC \to (\bB)_\bA^k$. The age $\cF$ has the \df{Ramsey property} if all structures in $\cF$ are Ramsey for $\cF$. We will say that an ultrahomogeneous structure $\bN$ has the Ramsey property if its age does.

The reason the Ramsey property is important in topological dynamics is the following fundamental result of Kechris, Pestov, and Todor\v{c}evi\'c.
\begin{theorem}[Kechris--Pestov--Todor\v{c}evi\'c~\cite{Kechris2005}]
  \label{th:KPT}
  Let $\bN$ be an ultrahomogeneous structure. Then the following are equivalent:
  \begin{itemize}
  \item $\bN$ has the Ramsey property;
  \item $\Aut(\bN)$ is extremely amenable.
  \end{itemize}
\end{theorem}
There is also a local version of the KPT theorem, which we will use. If $G$ is any topological group, $V \leq G$ is an open subgroup and $Z$ is a compact space, then $G$ acts naturally on $Z^{G/V}$ by
\begin{equation*}
  (g \cdot z)(hV) = z(g^{-1}hV).
\end{equation*}
We will say that $V \leq G$ is a \df{Ramsey subgroup of $G$} if the only minimal subflows of $2^{G/V}$ (equivalently, of $Z^{G/V}$ for any zero-dimensional compact $Z$) are the fixed points, that is, the constant functions $G/V \to 2$. Note that if $V \leq G$ is Ramsey and $\phi \in \Aut(G)$, then $\phi[V]$ is also Ramsey. 

Let now $\bN$ be ultrahomogeneous and let $G = \Aut(\bN)$. If $\bA$ is a finitely generated substructure of $\bN$, we will denote by $G_\bA$ the pointwise stabilizer of $\bA$ in $G$. This is an open subgroup of $G$ and it is well-known and easy to see that $G_{\bA}$ is a Ramsey subgroup of $G$ iff $\bA$ is a Ramsey substructure of $\bN$.
One equivalent way to state Theorem~\ref{th:KPT} is that a non-archimedean $G$ is extremely amenable iff every open subgroup of $G$ is Ramsey.

It is well-known that extreme amenability is stable under taking group extensions. Below we prove a local version of this fact.
\begin{prop}
  \label{p:local-group-ext}
  Let $G$ be a topological group, $H \lhd G$ be a closed, normal subgroup and let $K = G/H$. Denote by $\pi \colon G \to K$ the quotient map. Let $V \leq G$ be an open subgroup of $G$. If $V \cap H$ is Ramsey in $H$ and $\pi(V)$ is Ramsey in $K$, then $V$ is Ramsey in $G$.
\end{prop}
\begin{proof}
  Let $z_0 \in 2^{G/V}$. Our goal is to show that $\cl{G \cdot z_0}$ contains a fixed point. Let $T \sub G$ be a set of representatives for the double $(H, V)$-cosets in $G$, i.e., such that $G = \bigsqcup_{t \in T} HtV$. For $t \in G$, we denote by $V^{t^{-1}}$ the conjugate $tVt^{-1}$.
  There is a natural $H$-equivariant bijection  $\theta \colon G/V \to \bigsqcup_{t \in T} H/(V^{t^{-1}} \cap H)$ given by
  \begin{equation*}
    \theta(gV) = h(V^{t^{-1}}\cap H), \quad \text{ where } h \in H, t \in T, gV = htV,
  \end{equation*}
  which allows to identify $2^{G/V}$ and $\prod_{t \in T} 2^{H/(V^{t^{-1}} \cap H)}$, with the diagonal action, as $H$-spaces. Our assumption that $V \cap H$ is Ramsey in $H$ tells us that the set $Z$ of $H$-fixed points in $\cl{G \cdot z_0}$ is non-empty. 
  Indeed, the only minimal $H$-subflows of each factor $2^{H/(V^{t^{-1}} \cap H)}$ are the fixed points, so we can take an element of the product all of whose coordinates are fixed points.
  We can define an injective map $\phi \colon Z \to 2^{K/\pi(V)}$ by
  \begin{equation*}
    \phi(z)(\pi(g)\pi(V)) = z(gV).
  \end{equation*}
To conclude, note that any $K$-fixed point in $\phi(Z)$ gives us a $G$-fixed point in $Z$.
\end{proof}

Note also that $G$ is always a Ramsey subgroup of itself.
\begin{cor}
  \label{c:Ramsey-products}
  Let $\set{G_i : i \in I}$ be a family of topological groups and let $G = \prod_{i \in I} G_i$. Let $V_i \leq G_i$ be an open Ramsey subgroup for each $i \in I$, with $V_i = G_i$ for all but finitely many $i$. Then $\prod_i V_i$ is a Ramsey subgroup of $G$.
\end{cor}
\begin{proof}
  The case where $I$ is finite follows by induction on $|I|$ from Proposition~\ref{p:local-group-ext}. For an arbitrary $I$, let $F = \set{i \in I : V_i \neq G_i }$ and apply Proposition~\ref{p:local-group-ext} to the extension
  \begin{equation*}
    1 \to \prod_{i \in F} G_i \to G \to \prod_{i \in I \sminus F} G_i \to 1. \qedhere
  \end{equation*}
\end{proof}

As elements of the age of $\pkf(\bN)$ are rooted trees, to prove Theorem~\ref{th:main-Ramsey}, we will be led to do inductive arguments on the height of the tree.
In order to facilitate this, we introduce an auxiliary class of structures.

Let $\cL_0$ be a relational signature and let $\cc$ be a constant symbol. Let $\bN$ be an ultrahomogeneous structure in the signature $\cL_0 \cup \set{\cc}$ and denote $\Delta = \Aut(\bN)$.
Let $\pkf(\bN)$ be constructed as in Section~\ref{sec:kaleidoscopic_structures}, so $\Aut(\pkf(\bN)) = \pkf(\Delta)$. 
Denote $\cN = \Age(\bN)$ and $\pkf(\cN) = \Age(\pkf(\bN))$, and let $\cL^*$ be the language of $\pkf(\bN)$. The main goal of this section is to show that if $\cN$ has the Ramsey property, then so does $\pkf(\cN)$. An important role in the proof is played by the following notion. If $\bT, \bS \in \pkf(\cN)$, an embedding $f \colon \bT \to \bS$ is called \df{rooted} if $f(r_\bT) = r_\bS$. 
We can introduce a new constant $\rr$ to the language and consider the class $\pkf_\rr(\cN)$ of $\cL^* \cup \set{\rr}$-structures $\bT_\rr$, obtained from $\bT \in \pkf(\cN)$  by interpreting $\rr$ as $r_{\bT}$.

Fix some $r \in \bps$, and let $C_\xi$ denote the component $\rho(r)$ considered as a subset of $X$. Let $\pkf_\rr(\bN)$ denote the substructure of $\pkf(\bN)$ whose domain is $(\bps \setminus C_\xi) \cup \set{\xi}$ with the new constant $\rr$ interpreted as $r$. In order to study the structure $\pkf_\rr(\bN)$ and its automorphism group, we will repeatedly make use of the following lemma from \cite{Duchesne2019a}, which allows us to patch automorphisms from pieces.

\begin{lemma}[\cite{Duchesne2019a}*{Lemma~2.3}]
  \label{l:patchwork}
  Let $\cU$ be a family of disjoint open subsets of a dendrite $D$. For each $U \in \cU$, let $f_U$ be a homeomorphism of $D$ which is the identity outside of $U$. Then the map $f \colon D \to D$ given by $f_U$ on each $U$ and the identity outside of $\bigcup \cU$ is a homeomorphism of $D$.
\end{lemma}

An immediate consequence of Lemma~\ref{l:patchwork} is that $\Aut(\pkf_\rr(\bN))$ can be identified with the pointwise stabilizer of $C_\xi$ in $\pkf(\Delta)$. The following are the basic properties of $\pkf_\rr(\bN)$ which we will use.
\begin{prop}
  \label{p:pkfrrN}
  Let $\pkf_\rr(\bN)$ be defined as above. Then the following hold:
  \begin{enumerate}
  \item \label{i:pkfrr:ultrahom} $\pkf_\rr(\bN)$ is ultrahomogeneous and does not depend on the choice of $r \in \bps$, up to conjugation by an element of $\pkf(\bN)$;
  \item \label{i:pkfrr:age} $\Age(\pkf_\rr(\bN)) = \pkf_\rr(\cN)$;
  \item \label{i:pkfrr:cofinal} $\set{\bA[h+1]_{\rr} : \bA \in \cN, h \in \N}$ is cofinal in $\pkf_\rr(\cN)$.
  \end{enumerate}
\end{prop}
\begin{proof}
  \ref{i:pkfrr:ultrahom} Let $f \colon \bT \to \bT'$ be an isomorphism between finite substructures of $\pkf_\rr(\bN)$. Then, by ultrahomogeneity of $\pkf(\bN)$, there is $g \in \Aut(\pkf(\bN))$ such that $g|_\bT = f$; in particular, $g(r)=r$.
  We can then use Lemma~\ref{l:patchwork} to obtain an automorphism of $\pkf_\rr(\bN)$ by patching $g|_{X \sminus (C_\xi \cup \set{r})}$ with the identity on $C_\xi$.
  The second statement is a consequence of the transitivity of the action $\pkf(\bN) \actson X$, which follows from Theorem~\ref{thm:point-kal-is-homogeneous}.

  \ref{i:pkfrr:age} This is clear.

  \ref{i:pkfrr:cofinal} This follows from Proposition~\ref{prop:cofinal-in-point-kal} and \ref{i:pkfrr:age}.
\end{proof}

Now we are able to state one of the main ingredients used in the proof of the Ramsey theorem. 
\begin{lemma}
  \label{l:RootedRamsey}
  Let $\bA \in \cN$ and $h \in \N$. 
  If $\bA[h]$ is Ramsey in $\pkf(\cN)$ and $\bA$ is Ramsey in $\cN$ then $\bA[h+1]_\rr$ is Ramsey in $\pkf_\rr(\cN)$. 
\end{lemma}
\begin{proof}
  To simplify notation, let $G = \Aut(\pkf_\rr(\bN))$, $H = \Aut(\pkf(\bN))$, and recall that $\Delta = \Aut(\bN)$.
  For any $a \in N \setminus \set{c}$, denote by $C_a$ the set of points in $X$ in the component $(\pkc_r)^{-1}(a)$ and let $H_{(a)}$ be the pointwise stabilizer of $\bps \sminus C_a$ in $H$. Note that, by the remarks preceding \Cref{cor:homo-from-open-subgroup-rooted}, $H_{(a)}$ is isomorphic to $H$. 
  By ultrahomogeneity of $\bN$ and $\pkf(\bN)$, the hypotheses translate to $\Delta_A$ and $H_{\bA[h]}$ being Ramsey subgroups of $\Delta$ and $H$, respectively.
  
  We have the following short exact sequence: 
\begin{equation}
    \label{eq:pkfrr}
    1 \to \prod_{a \in N \setminus \set{c}}H_{(a)} \to G \overset{\pi}{\to} \Delta \to 1,  
\end{equation}
where $\pi(g) = \alpha_{\pkc}(g,r)$. Indeed, it is clear that $\pi$ is a homomorphism and it is surjective by Corollary~\ref{cor:homo-from-open-subgroup-rooted}. If we let $K = \ker \pi$, we can define a homomorphism $\theta_a \colon K \to H_{(a)}$ by $\theta_a(g) = g|_{C_a} \cup \id|_{X \sminus C_a}$. It is clear that the product $\prod_{a \in N \sminus \set{c}} \theta_a$ is injective and it is onto $\prod_{a \in N \setminus \set{c}}H_{(a)}$ by Lemma~\ref{l:patchwork}.

Recall the isomorphism $\lambda_\rr$ from \Cref{df:nice-A-h}. 
Realize $\bA[h+1]_\rr$ as a substructure of $\pkf_\rr(\bN)$, with $\pkc_r \lambda_\rr^{-1}(a) = a$ for all $a \in \bA$. 
Consider the open subgroup $V \coloneqq G_{\bA[h+1]_\rr}$ of $G$. 
Then 
\[
  U \coloneqq V \cap \prod_{a \in N \setminus \set{c}}H_{(a)} = \prod_{a \in N \setminus \set{c}} (H_{(a)} \cap V).
\]
This follows from the fact that for all $a$, $\theta_a(V \cap K) = H_{(a)} \cap V$. Note that $H_{(a)} \cap V = H_{(a)}$ for all $a \in M \setminus A$. For $a \in A \sminus \set{c}$, the $\cL^*$-substructure of $\bA[h+1]_\rr$ with domain $C_a \cup \set{\xi}$ is isomorphic to $\bA[h]$, so $H_{(a)} \cap V$ is conjugate to $H_{\bA[h]}$, which is a Ramsey subgroup of $H$, by hypothesis. Thus, by Corollary~\ref{c:Ramsey-products}, $U$ is a Ramsey subgroup of $K$. Also, $\pi[V] = \Delta_A$, so $\pi[V]$ is a Ramsey subgroup of $\Delta$. Thus, we can apply Proposition~\ref{p:local-group-ext} and conclude that $V$ is a Ramsey subgroup of $G$, i.e.,
that $\bA[h+1]_\rr$ is Ramsey in $\pkf_\rr(\cN)$. 
\end{proof}

\begin{remark}
  \label{rem:wreath-product}
  To understand better the group $\Aut(\pkf_\rr(\bN))$, as well as our strategy for proving the Ramsey theorem below, it will be perhaps helpful to observe that the exact sequence \eqref{eq:pkfrr} splits (by Corollary~\ref{cor:homo-from-open-subgroup-rooted}), so $\Aut(\pkf_\rr(\bN))$ is in fact isomorphic to the permutational wreath product
  \begin{equation*}
    \Aut(\pkf(\bN))^{N \sminus \set{c}} \wr \Delta.
  \end{equation*}
\end{remark}

The main theorem of this section is the following.
\begin{theorem}
  \label{th:main-Ramsey}
  Let $\cL_0$ be a relational signature and let $\cc$ be a constant symbol. Let $\bN$ be an ultrahomogeneous $\cL_0 \sqcup \set{\cc}$-structure with the Ramsey property. Then $\pkf(\bN)$ has the Ramsey property.
\end{theorem}

Combining this with Theorem~\ref{th:KPT}, we obtain:
\begin{cor}
  \label{cor:KstarDeltaExtAmenable}
  Let $\Delta \actson M$ be a permutation group, let $c \in M$, and suppose that $\Delta$ stabilizes $c$. If $\Delta$ is extremely amenable, then so is $\pkf(\Delta)$.
\end{cor}
\begin{proof}[Proof of \Cref{th:main-Ramsey}]
  We denote
  \begin{equation*}
    \cF = \set{\bA[h] : \bA \in \Age(\bN), h = 0, 1, \ldots}.
  \end{equation*}
  By Proposition~\ref{prop:cofinal-in-point-kal}, $\cF$ is cofinal in $\Age(\pkf(\bN))$.
  By ultrahomogeneity of $\pkf(\bN)$, it thus suffices to show that each $\bA[h]$ is a Ramsey structure in $\cF$. We do so by fixing $\bA$ and carrying out induction on $h$.
  
  For the base case, we prove that $\nice{\bA}{0}$ is Ramsey. To that end, take any $\nice{\bB}{h} \in \cF$. Notice that, for any $\bT \in \cF$, coloring the embeddings of $\nice{\bA}{0}$ in $\bT$ corresponds to coloring the points of $\bT$. Moreover, since $\bB$ is rigid, embeddings from $\nice{\bB}{h}$ to $\nice{\bB}{k}$ correspond to subtrees of $\nice{\bB}{k}$ isomorphic to $\nice{\bB}{h}$. Therefore, this is a statement about trees, and we can apply \cite{Deuber1975}*{Lemma~5} to get that $\nice{\bB}{2h-1} \to (\nice{\bB}{h})_2^{\nice{\bA}{0}}$.

  Next we do the induction step. Let $\ell > 0$ and suppose that $\nice{\bA}{\ell-1}$ is Ramsey. Let $\bU \in \cF$ be arbitrary such that $\nice{\bA}{\ell} \sub \bU$; we will find $\bT \in \cF$ such that $\bT \to (\bU)^{\nice{\bA}{\ell}}_2$ holds.

  We denote by $\tort$ the arrow relation for rooted embeddings.  It follows from \Cref{l:RootedRamsey} and the induction hypothesis that $\nice{\bA}{\ell}$ is Ramsey for rooted embeddings. 
  
  By the base of induction we know that there is $\nice{\bD}{h} \in \cF$ such that $\nice{\bD}{h} \to (\bU)^{\nice{\bA}{0}}_2$. Let $\bC_h \coloneqq \bD$ and $n_h \coloneqq 0$. Suppose we have defined $\bC_{i+1} \in \cN$ and $n_{i+1} \in \N$, for some $0 \le i < h$. Let $\bC_i \supseteq \bC_{i+1}, n_i \ge n_{i+1}$ be such that
  \begin{equation*}
    \nice{\bC_i}{n_i} \tort (\nice{\bC_{i+1}}{n_{i+1} +1})_2^{\nice{\bA}{\ell}}.
  \end{equation*}
  Let $\bT \coloneqq \nice{\bC_0}{n_0}$. We show that $\bT \to (\bU)_2^{\nice{\bA}{\ell}}$. To that end, fix a coloring $\gamma \colon \binom{\bT}{\nice{\bA}{\ell}} \to 2$. Let $\bp = r_{\nice{\bA}{\ell}}$.
  We construct a substructure $\bS$ of $\bT$ such that any two embeddings $f, f' \colon \nice{\bA}{\ell} \to \bS$ that agree on $\bp$ receive the same color. 
  
  Let $\bS_0 \coloneqq \bT$. 
  For $0 \le i < h$, suppose we have found $\bS_i \sub \bT$ such that
  \begin{enumerate}
      \item \label{itm:1}$\bS_i$ has height $i+n_i$;
      \item \label{itm:2}if $0 \le j<i$, then $\bS_i \restr{j} = \bS_j \restr{j}$;
      \item \label{itm:3}for any $\bs \in \bS_i$ of height $i$, $\bS_i(\bs) \cong \nice{\bC_i}{n_i}$; 
      \item \label{itm:4}for any $f, f' \colon \nice{\bA}{\ell} \to \bS_i$ such that $f(\bp) = f'(\bp)$ has height less than $i$, it holds that $\gamma(f) = \gamma(f')$. 
  \end{enumerate}
  It is clear that $\bS_0$ satisfies the conditions. We suppose that $\bS_i$ has been constructed and we show how to construct $\bS_{i+1} \sub \bS_i$. First, we let $\bS_{i+1} \restr{i}  = \bS_{i} \restr{i}$, which ensures that $\bS_{i+1}$ satisfies point \ref{itm:2} above. 
  Now fix $\bs$ of height $i$ in $\bS_i$. 
  Since $\bS_i(s) \cong \nice{\bC_i}{n_i}$, we have that $\bS_i(s) \overset{r}{\to} (\nice{\bC_{i+1}}{n_{i+1} +1})_2^{\nice{\bA}{\ell}}$, so there is a $\gamma$-monochromatic rooted embedding of $\nice{\bC_{i+1}}{n_{i+1} +1}$ into $\bS_i(s)$. 
  We let $\bS_{i+1}(s)$ be the image of this embedding, so $\bS_{i+1}$ satisfies point \ref{itm:3}. 
  Then $\bS_{i+1}$ has height $i+ n_{i+1} + 1$, so point \ref{itm:1} is taken care of. 
  Suppose that $f, f' \colon \nice{\bA}{\ell} \to \bS_{i+1}$ are such that $\bs \coloneqq f(\bp) = f'(\bp)$ has height $j \le i$. 
  If $j<i$, then $\bs$ has height $j$ in $\bS_j$ by point \ref{itm:2}, so $\gamma(f) = \gamma(f')$ by \ref{itm:4} for $\bS_j \supseteq \bS_i$. 
  Otherwise, $f$ and $f'$ are rooted embeddings of $\nice{\bA}{\ell}$ into $\bS_{i+1}(s)$, so by construction, they receive the same $\gamma$-color. 
  This takes care of point \ref{itm:4} and completes the inductive construction. 

  Let $\bS = \bS_h$. 
  By \ref{itm:1} it has height $h + n_h = h$. 
  By points \ref{itm:2} and \ref{itm:3}, for any $\bs \in S$ of height less than $h$, $\nice{\bD}{1}$ embeds into $\bS(s) \restr{1}$, so $\nice{\bD}{h}$ embeds into $\bS$.
  It follows that $\bS \to (\bU)_2^{\nice{\bA}{0}}$.  
  By point \ref{itm:4}, for each $\bs \in S$ there is $\delta(s) \in \set{0, 1}$ such that $\gamma(f) = \delta(s)$ for each $f \colon \nice{\bA}{\ell} \to \bS$ such that $f(\bp)=\bs$. (Note that, as $\ell \geq 1$, there is no embedding $f \colon \bA[\ell] \to \bS$ with $f(p)$ of height $h$.) By recalling that coloring points is equivalent to coloring embeddings of $\nice{\bA}{0}$, we find a copy $\bU_0$ of $\bU$ in $\bS \subseteq \bT$ which is $\delta$-monochromatic. It follows that $\bU_0$ is $\gamma$-monochromatic as well and we are done.
\end{proof}

\begin{remark}
  It follows from the representation given in Remark~\ref{rem:wreath-product} that if $\Delta$ is extremely amenable, then so is $\Aut(\pkf_\rr(\bN))$.
\end{remark}


\section{The universal highly proximal extension of the action on \texorpdfstring{$W$}{W}}
\label{sec:maxim-highly-prox}

As a first application of the tools we have developed, we give a description of the universal highly proximal extension of the flow $\kf(\Gamma) \actson W$ and characterize when it is metrizable. 

\begin{theorem}
  \label{th:S_GW}
  Let $\Gamma \actson M$ be transitive and let $\bM$ be an ultrahomogeneous, relational structure for $\Gamma$ with universe $M$. Then
  \begin{equation}
    \label{eq:MHP_W}
    \MHP_{\kf(\Gamma)}(W) \cong \set{p \in \tS_1(\fullkf(\bM)) : \forall x, y \in X \ p(v) \models \neg B(x, v, y)},
  \end{equation}
  with the action of $\kf(\Gamma)$ on the right-hand side given by permuting the parameters. The highly proximal factor map $\pi \colon \MHP_{\kf(\Gamma)}(\waz) \to \waz$ is given by
  \begin{equation}
    \label{eq:SGW-factor-map}
  \pi(p) = \zeta \iff \forall x \in \bps \setminus \set{\zeta}\ p(v) \models \compf(x, v) = \compf(x, \zeta).
\end{equation}
\end{theorem}
\begin{proof}
  Write $G = \kf(\Gamma)$ and let $\xi \in W$ be an endpoint. By Fact~\ref{f:kaleidoscopic-facts}, the orbit $G \cdot \xi = \theends(W)$ is comeager in $W$, so from Proposition~\ref{p:SG-comeager}, we have that $\MHP_G(W) \cong \Sam(G/G_\xi)$.  Let $\cc$ be a new constant symbol, fix some $c \in M$ and let $\bM_\cc$ be the expansion of $\bM$ where one interprets $\cc$ as $c$, as in Subsection~\ref{sec:endpoint-stabilizer}. Proposition~\ref{p:point-stabilizer} gives us a representation of $G_\xi$ as the automorphism group $\Aut(\pkf(\bM_\cc))$. Denote by $\cl{\cL}$ the qf-full signature of $\fullkf(\bM)$ and let $\cl{\cL}_{\bxi} \coloneqq \cl{\cL} \sqcup \set{\bxi}$. By Lemma~\ref{l:fullkf-fullkpf}, $\fullpkf(\bM_\cc)$, the expansion of $\pkf(\bM_\cc)$ to $\cl{\cL}_{\bxi}$, is full.
  Now we can apply Proposition~\ref{p:SGH-as-types} to obtain that
  \begin{equation}
    \label{eq:S_MHP_W2}
    \MHP_G(W) \cong \Exp(\kf(\bM), \fullpkf(\bM_\cc)).
  \end{equation}
 Let $\cS$ and $\cS'$ denote the right-hand sides of \eqref{eq:S_MHP_W2} and \eqref{eq:MHP_W}, respectively. Next we explain how to identify $\cS$ and $\cS'$.
 We define the map $\Theta \colon \cS \to \tS_1(\fullkf(\bM))$ by
  \begin{equation*}
    \Theta(p)(v) \models \phi(v, \bar x) \iff p \models \exists z \in \rho(\bigwedge \bar x) \ \phi(z, \bar x) 
  \end{equation*}
  for every quantifier-free $\cl{\cL}$-formula $\phi$.

  That $\Theta$ is well-defined follows from the fact that $\fullkf(\bM)$ eliminates quantifiers in $\cl{\cL}$. It is obvious that $\Theta$ is continuous and $G$-equivariant. To see that it is injective, note that if $p_1 \neq p_2 \in \cS$, then the proof of Lemma~\ref{l:fullkf-fullkpf} and \eqref{eq:equiv-Lstar} give us a quantifier-free $\cl{\cL}$-formula $\phi$ which witnesses that $\Theta(p_1) \neq \Theta(p_2)$. $\Im \Theta \sub \cS'$ because the sentence
  \begin{equation*}
    \forall x_1, x_2 \ \forall z \in \rho(x_1 \meet x_2) \ \neg B(x_1, z, x_2)
  \end{equation*}
  holds in $\fullpkf(\bM_\cc)$, so $\Theta(p)(v) \models \neg B(x_1, v, x_2)$. 
  Next we see that if $\phi(v, \bar x)$ is an isolating, quantifier-free $\cl{\cL}$-formula and $\exists v \ \phi(v, \bar x) \land \bigwedge_{i, j} \neg B(x_i, v, x_j)$ is true in $\fullkf(\bM)$, then the formula  $\exists z \in \rho(\bigwedge \bar x) \ \phi(z, \bar x)$ is consistent with $\Th(\fullpkf(\bM_\cc))$. Notice that for such a $v$ there is a single component $a \in \comps{v}$ such that $\bar x$ is contained in $a$.
  Now, it suffices to choose as $\bxi$ an endpoint in a component around $v$ which is not $a$ (and recall that for all choices of an endpoint $\xi$ the obtained models are isomorphic). 
  Indeed, for any such $\xi$, we have $v \in \rho(\bigwedge \bar x)$, so we can take $z = v$. 
  This shows $\Im \Theta$ is dense in $\cS'$, so by compactness, it is all of it.

  It remains to check the final assertion of the theorem. First, we see that \eqref{eq:SGW-factor-map} indeed defines a function. If $\zeta_1 \neq \zeta_2$, then there exists $x \in X \cap (\zeta_1, \zeta_2)$, so $\Phi(x, \zeta_1) \neq \Phi(x, \zeta_2)$. That $\pi$ is continuous follows from the fact that the sets of the form $\set{\zeta \in W : \Phi(x, \zeta) = a}$, where $x$ varies in $X$ and $a$ varies in $\comps x$, form a subbasis for the topology of $W$. That $\pi$ is $G$-equivariant is obvious. Finally the uniqueness of the factor map (Proposition~\ref{p:SG-comeager}) implies that $\pi$ is indeed the highly proximal factor map.
\end{proof}

\begin{cor}
  \label{c:metrizability-SGW}
  Let $\Gamma \actson M$ be transitive. Then the following are equivalent:
  \begin{enumerate}
  \item \label{i:thSGW:metrizable} $\MHP_{\kf(\Gamma)}(W)$ is metrizable;
  \item \label{i:thSGW:olig} $\Gamma$ is oligomorphic;
  \item \label{i:thSGW:coprecmop} $\kf(\Gamma)_\xi$ is co-precompact in $\kf(\Gamma)$.
  \end{enumerate}
\end{cor}
\begin{proof}
  \begin{cycprf}
  \item[\impnext] Let $\pi \colon \MHP_{\kf(\Gamma)}(W) \to W$ be the map given by \eqref{eq:SGW-factor-map}. Let $\bM$ be a structure with universe $M$ as in Theorem~\ref{th:S_GW}. Fix $x_0 \in X$ and consider the map $\Psi \colon \pi^{-1}(x_0) \to \tS_1(M)$ given by
  \begin{equation}
      \label{eq:map-to-S1M}
    \Psi(p)(v) \models \theta(v, \bar a) \iff p(v) \models \theta(\compf(x_0, v), \kc_{x_0}^{-1}(\bar a)),
  \end{equation}
  for any quantifier-free formula $\theta$ in the language of $\bM$ and any tuple $\bar a$ from $M$. It is clear that $\Psi$ is continuous. It is also surjective because the realized types are dense in $\tS_1(M)$ and they belong to the image of $\Psi$. If $\Gamma \actson M$ is not oligomorphic, then $\tS_1(M)$ is not metrizable by Corollary~\ref{p:samuel-open}, so $\MHP_{\kf(\Gamma)}(W)$ is not metrizable either.
  
\item[\impnext] If $\Gamma$ is oligomorphic, then $\Gamma_c \actson M$ is oligomorphic, so, by \Cref{c:KstarDeltaOligo}, $\pkf(\Gamma_c) \actson X$ is also oligomorphic. From Proposition~\ref{p:point-stabilizer} \ref{i:ps:2}, we have that
  $\kf(\Gamma)_\xi = \pkf(\Gamma_c)$, so $\kf(\Gamma)_\xi \actson X$ is oligomorphic. Now \ref{i:thSGW:coprecmop} follows from the general fact that an oligomorphic subgroup of any permutation group is always co-precompact.
  
\item[\ref{i:thSGW:coprecmop} $\Leftrightarrow$ \ref{i:thSGW:metrizable}] Since endpoints in $W$ are a comeager orbit for the action of $\kf(\Gamma)$, we have that $\MHP_{\kf(\Gamma)}(W) = \Sam(\kf(\Gamma)/\kf(\Gamma)_\xi)$, which is metrizable if and only if $\kf(\Gamma)/\kf(\Gamma)_\xi$ is precompact.
\end{cycprf}
\end{proof}

Since $\MHP_{\kf(\Gamma)}(W)$ is minimal, this shows that if $\Gamma$ is not oligomorphic then $\kf(\Gamma)$ is not CAP. 

\begin{example}
  \label{ex:3-transitive}
  As an illustration of Theorem~\ref{th:S_GW}, we give a description of the preimages $\pi^{-1}(\zeta)$ (where $\pi \colon \MHP_{\kf(\Gamma)}(W) \to W$ is the factor map given by \eqref{eq:SGW-factor-map}) for the different types of points $\zeta \in W$ in the case when the action $\Gamma \actson M$ is $3$-transitive. 
  Let $\bM$ be a qf-full structure for $\Gamma$ with universe $M$.
  We show that the preimages of endpoints and regular points are singletons and for each $x \in \bps$, we have $\pi^{-1}(x) \cong \tS_1(M)$. Recall that $\tS_1(M)\cong \Sam(\Gamma / \Gamma_c)$, for some $c \in M$, by \Cref{p:samuel-open}. 

  If $\xi \in \waz$ is an endpoint, $\pi^{-1}(\xi)$ is always a singleton, by \Cref{p:SG-comeager}. 

  Now fix a regular point $\zeta \in \waz$. We show that $\pi^{-1}(\zeta)$ is also a singleton. Let $p \in \tS_1(\kf(\bM))$ be such that $\pi(p) = \zeta$ and let $T \sub X$ be a finite subtree such that $\zeta$ belongs to the subdendrite generated by $T$. Let $x_0, x_1 \in T$ be such that $\zeta \in (x_0, x_1)$ and there are no vertices of $T$ in $(x_0, x_1)$. Let $y$ be a realization of $p|_T$ and let $z = K(x_0, x_1, y)$. Note that the tree generated by $T$ and $y$ is $T' \coloneqq T \cup \set{y, z}$. By quantifier elimination, the type $\tp(y/T)$ is determined by the isomorphism type of the substructure $\bT' \sub \kf(\bM)$ with domain $T'$. The isomorphism type of $\bT'$ as a tree (in the language $\cL_\bX$) is already determined by $\zeta$ and the rest is decided by
  \begin{equation}
    \label{eq:type-zeta}
    \tp_{\bM}(\kappa(\Phi(z, y)), \kappa(\Phi(z, x_0)), \kappa(\Phi(z, x_1))),
  \end{equation}
  which is unique by the $3$-transitivity assumption. 

  Now consider a branching point $x_0 \in \waz$. 
  Let the map $\Psi \colon \pi^{-1}(x_0) \to \tS_1(M)$ be defined by \eqref{eq:map-to-S1M}. We will show that if $\Gamma$ is $3$-transitive, then $\Psi$ is a homeomorphism. We only have to prove that $\Psi$ is injective. 
  We do so by showing that $\Psi^{-1}(q)$ is a singleton for any $q \in \tS_1(M)$. Let $p$ be any element of $\Psi^{-1}(q)$. We have two cases.

  \textbf{Case 1.} The type $q$ is realized in $\bM$ by an element $a_0$. Let $T$ be a finite subtree of $X$ containing $x_0$ and having a point in the component $\kappa_{x_0}^{-1}(a_0)$. Let $x_1 \in T$ be the point in this component closest to $x_0$. Take a realization $y$ of $p|_T$ and let $z = K(x_0, x_1, y)$. Then the substructure of $\kf(\bM)$ with domain $T' \coloneqq T \cup \set{y, z}$ again has isomorphism type determined by $a_0$ and the $3$-type \eqref{eq:type-zeta}. So we conclude by $3$-transitivity as before.

  \textbf{Case 2.} The type $q$ is not realized in $\bM$. Let again $T$ be a finite subtree of $X$ containing $x_0$ and let $y$ be a realization of $p|_T$. Now $T \cup \set{y}$ is already closed under $K$, so the isomorphism type of the substructure of $\kf(\bM)$ with domain $T\cup \set{y}$ is determined by $\tp_\bM(\kappa_{x_0}^{-1}(\Phi(x_0, y)) / \kappa_{x_0}^{-1}(\comps x_0  \cap \Comps{T}))$, which is decided by $q$.

  We conclude with two remarks about the analysis above. First, one can describe $\MHP_{\kf(\Gamma)}(W)$ even without the $3$-transitivity assumption but the description is more complex as it has to involve the space of $3$-types $\tS_3(\Th(\bM))$. Second, one obtains a particularly simple description when $\kf(\Gamma)$ is the full homeomorphism group of $W$ (which will be used in Section~\ref{sec:case-full-home} below). Namely, in that case, $\tS_1(M)$ can simply be identified with $M$ when $M$ is finite and with the one-point compactification of $M$ when $M$ is infinite.
\end{example}


\section{Universal minimal flows of kaleidoscopic groups}
\label{sec:univ-minim-flows}


\subsection{The UMFs of endpoint stabilizers}
\label{sec:univ-minim-stabilizer}

With the Ramsey theorem from Section~\ref{sec:ramsey-interlude} in hand, we start by describing the UMFs of the endpoint stabilizer $\kf(\Gamma)_\xi$.
\begin{theorem}
  \label{th:KGammaxiComeagerorbit}
  Let $\Gamma \acts M$ be transitive and let $c \in M$.
  If $\umf(\Gamma_c) = \Sam(\Gamma_c / \Delta)$ for some closed $\Delta \le \Gamma$, then $\umf(\kf(\Gamma)_\xi) = \Sam(\pkf(\Gamma_c) / \pkf(\Delta))$. In particular, if $\umf(\Gamma_c)$ has a comeager orbit, so does $\umf(\kf(\Gamma)_\xi)$.
\end{theorem}
\begin{proof}
  By Fact~\ref{f:facts-umf}, $\Delta$ is extremely amenable and presyndetic in $\Gamma_c$. 
  By \Cref{cor:KstarDeltaExtAmenable}, $\pkf(\Delta)$ is extremely amenable and by \Cref{pr:H-presyndetic}, it is presyndetic in $\kf(\Gamma)_\xi = \pkf(\Gamma_c)$. So $\umf(\kf(\Gamma)_\xi) = \Sam(\pkf(\Gamma_c) / \pkf(\Delta))$, by Fact~\ref{f:facts-umf}. In particular, we have the statement about the comeager orbits.
\end{proof}

When $\Gamma_c$ is moreover CAP, we have a more explicit description of $\umf(\kf(\Gamma)_\xi)$.
  
\begin{theorem}
  \label{th:KGammaxiCAPUMF}
  Let $\Gamma \acts M$ be transitive and let $c \in M$. If $\Gamma_c$ is CAP, we have that:
  \begin{equation*}
    \umf(\kf(\Gamma)_\xi) \cong \umf(\Gamma_c)^X,
  \end{equation*}
  with the action $\kf(\Gamma)_\xi \acts \umf(\Gamma_c)^\bps$ given by
  \begin{equation*}
    (g \cdot q)(x) = \alpha_{\pkc}(g, g^{-1} \cdot x) \cdot q(g^{-1} \cdot x),
  \end{equation*}
  where $\pkc$ is a root-kaleidoscopic coloring and $\alpha_{\pkc}$ is the cocycle defined in Section~\ref{sec:kale-root-kale}.
\end{theorem}
\begin{proof}
  Let $\bM$ be a structure on $M$ such that $\Gamma = \Aut(\bM)$ and let $c \in M$. Let $\Delta \leq \Gamma_c$ be such that $\umf(\Gamma_c) = \Sam(\Gamma_c / \Delta)$ and let $\bN$ be an ultrahomogeneous relational expansion of $(\bM, \cc)$ with $\Delta = \Aut(\bN)$. Let $\cL_0$ be the signature of $\bM$ and $\cL_1 \sqcup \set{\cc}$ be the signature of $\bN$. Our assumption that $\Gamma_c$ is CAP, that is, that $\umf(\Gamma_c) = \Sam(\Gamma_c / \Delta)$ is metrizable, translates to the fact that we can take $\cL_1$ such that $\cL_1 \sminus \cL_0$ contains finitely many relation symbols in each arity. Now construct the structures $\pkf(\bM_{\cc})$ and $\pkf(\bN)$ as in Section~\ref{sec:kaleidoscopic_structures} and recall that $\pkf(\Gamma_c) = \Aut(\pkf(\bM_\cc))$ and $\pkf(\Delta) = \Aut(\pkf(\bN))$. Let $\fullpkf(\bM_\cc)$ be a qf-full structure for $\Aut(\pkf(\bM_\cc))$ and denote by $\cl{\cL}_{\bxi}$ its signature.
  \begin{claim*}
    The $\cl{\cL} \sqcup (\cL_1 \sminus \cL_0)$-expansion of $\fullpkf(\bM_\cc)$ in which the relations in $\cL_1 \sminus \cL_0$ are interpreted on the imaginary sort $\cps$ as usual is qf-full for $\Aut(\pkf(\bN))$.
  \end{claim*}
  \begin{proof}
    Let $D \sub \pkf(\bN)^k$ be an $\Aut(\pkf(\bN))$-invariant set and write $D = \bigcup D_i$, where each $D_i$ is an $\Aut(\pkf(\bN))$-orbit. By ultrahomogeneity of $\pkf(\bN)$, there exists an isolating formula $\phi_i(\bar v)$ which describes the isomorphism type of tuples in $D_i$ and we can write $\phi_i$ in the form $\phi'_i \land \psi_i$, where $\phi_i'$ is an $\cl{\cL}$-formula and $\psi_i$ is a conjunction of relations in $\cL_1 \sminus \cL_0$. As there are only finitely many such relations of arity $\leq k$, there are only finitely many choices for the $\psi_i$, say $\psi_{i_1}, \ldots, \psi_{i_m}$. Now we have that
    \begin{equation*}
      \begin{split}
        \bar x \in D &\iff \bar x \in \bigcup_i D_i \\
        &\iff \pkf(\bN) \models \bigvee_i (\phi_i'(\bar x) \land \psi_i(\bar x)) \\
        &\iff \pkf(\bN) \models \bigvee_{j \leq m} \big( \psi_{i_j}(\bar x) \land \bigvee_{\set{i : \psi_i = \psi_{i_j}}} \phi_i'(\bar x) \big).
      \end{split}
    \end{equation*}
    By qf-fullness of $\fullpkf(\bM_{\cc})$, each $\bigvee_{\set{i : \psi_i = \psi_{i_j}}} \phi_i'(\bar x)$ is equivalent to a quantifier-free $\cl{\cL}$-formula, so this proves the claim.
  \end{proof}
  Now we can apply Proposition~\ref{p:SGH-as-types} and obtain the representation
  \begin{equation}
    \label{eq:pkfGamma_c}
    \umf(\pkf(\Gamma_c)) = \Sam(\pkf(\Gamma_c) / \pkf(\Delta)) = \Exp(\pkf(\bM_\cc), \fullpkf(\bN)).
  \end{equation}
  Similarly,
  \begin{equation}
    \label{eq:Gamma_c}
    \umf(\Gamma_c) = \Sam(\Gamma_c / \Delta) = \Exp(\bM_\cc, \bN),
  \end{equation}
  We can define a map $\Theta \colon \umf(\pkf(\Gamma_c)) \to \umf(\Gamma_c)^X$ as follows:
  \begin{equation*}
    \Theta(p)(x) \models \phi(\bar a) \iff p \models \phi((\pkc)_x^{-1}(\bar a))
  \end{equation*}
  for all $x \in X$, $\bar a \in M^k$ and quantifier-free $\cL_1$-formulas $\phi$.
  In words, to determine $\Theta(p)$ in the $x$-th coordinate, we just look at the components around $x$ and copy the expansion using $\pkc$. It is clear that this map is continuous, surjective, and $\pkf(\Gamma_c)$-equivariant. Injectivity follows from the Claim.
\end{proof}

\begin{cor}
  \label{c:KGammaxiEAandCAP}
  Let $\Gamma \acts M$ be a transitive permutation group and let $c \in M$. Then:
  \begin{enumerate}
    \item \label{i:c:KGX:ea} $\kf(\Gamma)_\xi$ is extremely amenable iff $\Gamma_c$ is;
    \item \label{i:c:KGX:cap} $\kf(\Gamma)_\xi$ is CAP iff $\Gamma_c$ is.
  \end{enumerate}
\end{cor}
\begin{proof}
  The ($\Leftarrow$) direction of both \ref{i:c:KGX:ea} and \ref{i:c:KGX:cap} follows from Theorem~\ref{th:KGammaxiCAPUMF}.

  For the other direction, fix $x \in \bps$. 
  Then the stabilizer $\kf(\Gamma)_{\set{\xi, x}}$ is an open subgroup of  $\kf(\Gamma)_\xi$ and admits a surjective homomorphism to $\Gamma_c$ (cf. \Cref{cor:homo-from-open-subgroup-rooted}) so the ($\Rightarrow$) direction of both \ref{i:c:KGX:ea} and \ref{i:c:KGX:cap} follows from \Cref{f:ea-open-subgroup}.
\end{proof}


\subsection{The UMFs of kaleidoscopic groups}

By combining the results so far, we obtain a description of the UMF of $\kf(\Gamma)$.

\begin{theorem}
  \label{th:UMF-kal}
  Let $\Gamma \acts M$ be transitive and let $c \in M$. If $\umf(\Gamma_c) = \Sam(\Gamma_c/\Delta)$ for some closed $\Delta \leq \Gamma_c$, then
  \begin{equation*}
    \umf(\kf(\Gamma)) = \Sam(\kf(\Gamma)/\pkf(\Delta)).
  \end{equation*}
  In particular, if $\umf(\Gamma_c)$ has a comeager orbit, then so does $\umf(\kf(\Gamma))$.
\end{theorem}
\begin{proof}
  Recall that if $\umf(\Gamma_c)$ has a comeager orbit, then it has the form $\Sam(\Gamma_c / \Delta)$ for some  extremely amenable and presyndetic $\Delta \leq \Gamma_c$ (Fact~\ref{f:facts-umf}). By Corollary~\ref{cor:KstarDeltaExtAmenable}, $\pkf(\Delta)$ is extremely amenable and by Proposition~\ref{pr:H-presyndetic}, it is presyndetic in $\kf(\Gamma)$. Applying Fact~\ref{f:facts-umf} in the other direction, we obtain the conclusion.
\end{proof}

It is also possible to represent $\umf(\kf(\Gamma))$ as a space of expansions using Proposition~\ref{p:SGH-as-types}.
\begin{cor}
  \label{c:umf-KGamma-expansions}
  Let $\Gamma \actson M$ be transitive and $c \in M$. Suppose that the structure $\bM_\cc$ has a minimal, Ramsey, ultrahomogeneous expansion $\bN$. Then
  \begin{equation*}
    \umf(\kf(\Gamma)) = \Exp(\kf(\bM), \fullpkf(\bN)).
  \end{equation*}
\end{cor}

Using the results from the last section, we obtain a characterization of when $\kf(\Gamma)$ is CAP. 

\begin{theorem}
  \label{th:KGammaCAP}
  Let $\Gamma \acts M$ be a transitive permutation group. Then the following are equivalent:
  \begin{enumerate}
    \item \label{i:th:KGammaCAP-2} $\Gamma$ is CAP and the action $\Gamma \actson M$ is oligomorphic;
    \item \label{i:th:KGammaCAP-1} $\kf(\Gamma)$ is CAP.
  \end{enumerate}
\end{theorem}
\begin{proof}
  \begin{cycprf}
  \item[\impnext] If $\Gamma$ is CAP, then so is $\Gamma_c$ by \Cref{f:ea-open-subgroup}. 
  By \Cref{c:KGammaxiEAandCAP}, $\kf(\Gamma)_\xi$ is CAP, and by Corollary~\ref{c:metrizability-SGW}, it is co-precompact in $\kf(\Gamma)$.
  So we can conclude by point \ref{i:cor:mf-subgroups-coprecompact} of \Cref{cor:umf-subgroups}. 
  
  \item[\impfirst] Fix $x \in \bps$. Then the stabilizer $\kf(\Gamma)_{x}$ is an open subgroup of  $\kf(\Gamma)$ and admits a surjective homomorphism to $\Gamma$ by \Cref{f:kaleidoscopic-facts} \ref{i:kaleido-fact:surjective-homo}, so $\Gamma$ is CAP by \Cref{f:ea-open-subgroup}. Also, $\MHP_{\kf(\Gamma)}(W)$ is a minimal flow of $\kf(\Gamma)$, so it must be metrizable. Now Corollary~\ref{c:metrizability-SGW} implies that $\Gamma \actson M$ is oligomorphic.
  \end{cycprf}
\end{proof}

\begin{remark}
  \label{rem:open-subgroup}
  In view of \Cref{th:UMF-kal}, it may be more natural to replace the condition ``$\Gamma$ is CAP'' in \ref{i:th:KGammaCAP-2} above by ``$\Gamma_c$ is CAP''. However, for an oligomorphic $\Gamma$ they are equivalent.
\end{remark}

\begin{remark}
  We remark that the implication \ref{i:th:KGammaCAP-1} $\Rightarrow$ \ref{i:th:KGammaCAP-2} does not require $\Gamma$ to be transitive. Indeed, the argument showing that $\Gamma$ is CAP does not use transitivity. For oligomorphicity, one can show that if $\Gamma \actson M$ is not oligomorphic, then the minimal flow $\kf(\Gamma) \actson \MHP_{\kf(\Gamma)}(W)$ is not metrizable. Indeed, denote $G = \kf(\Gamma)$. If $\Gamma$ is not oligomorphic, it is easy to find a finite subset $F \sub X$ such that the set $G_F \backslash G / G_\xi$ is infinite and by Fact~\ref{f:SGH-non-archimedean}, $\MHP_G(W)$ admits a surjective map to $\beta(G_F \backslash G / G_\xi)$.
\end{remark}

A particularly simple situation is when the stabilizer $\Gamma_c$ is already extremely amenable (in particular, when $\Gamma$ is extremely amenable). Then we get the following.
\begin{cor}
  \label{c:Gamma_c-ea}
  Let $\Gamma \actson M$ be a transitive action and suppose that $\Gamma_c$ is extremely amenable. Then
  \begin{equation*}
    \umf(\kf(\Gamma)) = \MHP_{\kf(\Gamma)}(W).
  \end{equation*}
\end{cor}
\begin{proof}
  This follows from Theorem~\ref{th:UMF-kal} and \Cref{p:SG-comeager}, according to which $\MHP_{\kf(\Gamma)}(W) = \Sam(\kf(\Gamma)/\kf(\Gamma)_\xi)$.
\end{proof}


\subsection{The Furstenberg boundary of $\kf(\Gamma)$}
\label{sec:furst-bound-kfgamma}

Let $G$ be a topological group. Recall that a flow $G \actson X$ is called \df{strongly proximal} if for every Borel probability measure $\mu$ on $X$, there exists a net $(g_i)$ of elements of $G$ such that $g_i \cdot \mu$ converges to a Dirac measure in the weak$^*$-topology. Similarly to the universal minimal flow, every topological group $G$ admits a universal minimal, strongly proximal flow $\cF(G)$ called the \df{Furstenberg boundary of $G$}. $G$ is amenable iff $\cF(G)$ is trivial. As $\cF(G)$ is MHP, the general techniques of Zucker~\cite{Zucker2021} apply; in particular, he managed to give a characterization of when $\cF(G)$ has a comeager orbit similar to the one for $\umf(G)$. This, together with our previous results, allows us to compute $\cF(G)$ for some kaleidoscopic groups. A similar calculation was previously done by Duchesne~\cite{Duchesne2020} for the full homeomorphism group of $W$.

We start with an amenability result for the endpoint stabilizer.
\begin{theorem}
  \label{th:KGammaxiAmenable}
  Let $\Gamma \acts M$ be a transitive permutation group and let $c \in M$. Suppose that $\Gamma_c$ is CAP. Then $\kf(\Gamma)_\xi$ is amenable if and only if $\Gamma_c$ is.
\end{theorem}
\begin{proof}
  ($\Rightarrow$) This follows from \Cref{f:ea-open-subgroup} as in the proof of Corollary~\ref{c:KGammaxiEAandCAP}.

  ($\Leftarrow$) Write $G = \kf(\Gamma)$ and suppose $\Gamma_c$ is amenable, that is, $\umf(\Gamma_c)$ admits an invariant probability measure $\mu$. 
  By \Cref{th:KGammaxiCAPUMF}, $\umf(G_\xi) \cong \umf(\Gamma_c)^\bps$. 
  As in \cite{Duchesne2020}*{Proposition 8.4},
  our goal is to prove that the product measure $\mu^{\otimes \bps}$ on $\umf(\Gamma_c)^\bps$ is invariant under the action of $G_{\xi}$. We will check that the measure of cylinders is preserved. Let $F \sub \bps$ be finite and, for each $x \in F$, let $A_x \sub \umf(\Gamma_c)$ be a measurable set.
  Let $A = \prod_{x \in F} A_x \times \umf(\Gamma_c)^{X \setminus F}$. Then
  \begin{equation*}
    \begin{split}
      \mu^{\otimes X}(g \cdot A) &= \mu  \Big( \big( \prod_{y \in gF} \alpha_{\pkc}(g, g^{-1} \cdot y)  \cdot A_{g^{-1} \cdot y} \big) \times \umf(\Gamma_c)^{X \setminus gF} \Big) \\
      &= \prod_{y \in gF} \mu\big(\alpha_{\pkc}(g, g^{-1} \cdot y)  \cdot A_{g^{-1} \cdot y} \big) \\
      &= \prod_{x \in F} \mu(A_x) = \mu^{\otimes X}(A). \qedhere
    \end{split} 
  \end{equation*}
\end{proof}

\begin{theorem}
  \label{th:furstenberg-boundary}
  Let $\Gamma \acts M$ be a transitive permutation group and let $c \in M$. 
  If $\Gamma_c$ is amenable and $\Gamma_c$ is CAP, then $\cF(\kf(\Gamma)) \cong \MHP_{\kf(\Gamma)}(\waz)$.
\end{theorem}
\begin{proof}
  By \cite{Duchesne2018}*{Theorem 10.1}, $\kf(\Gamma) \acts \waz$ is a strongly proximal, minimal action, so $\kf(\Gamma)$ is not amenable. 
  By \Cref{pr:H-presyndetic}, \Cref{f:kaleidoscopic-facts} \ref{i:kaleido-fact:maximal-sub}, and \Cref{th:KGammaxiAmenable}, $\kf(\Gamma)_\xi$ is a presyndetic, maximal amenable closed subgroup of $\kf(\Gamma)$, so by \cite{Zucker2021}*{Theorem 7.5}, $\MHP_{\kf(\Gamma)}(\waz) = \Sam(\kf(\Gamma)/\kf(\Gamma)_\xi)$ is the Furstenberg boundary of $\kf(\Gamma)$. 
\end{proof}


\section{The case of the full homeomorphism group}
\label{sec:case-full-home}

When $\Gamma = \Sym(M)$, $\kf(\Gamma)$ becomes the full homeomorphism group of the dendrite $W$, which we denote by $G$. In this case, $\umf(G)$ was already calculated by Kwiatkowska~\cite{Kwiatkowska2018}. We start by explaining what Corollary~\ref{c:umf-KGamma-expansions} gives in this situation. First, $\bM$ is a structure in the language with only equality and $\kf(\bM) = \bX$. Let $c \in M$. Then a minimal Ramsey expansion $\bN$ of $\bM_\cc$ is given by:
\begin{itemize}
\item a linear order on $M$, where $c$ is the least element, in the case where $M$ is finite (in that case the group $\Aut(\bN)$ is the trivial group);
  
\item a linear order on $M$, where $c$ is the least element and the order on $M \sminus \set{c}$ is isomorphic to $(\bQ, <)$, in the case where $M$ is infinite (in that case, the group $\Aut(\bN)$ is isomorphic to $\Aut(\bQ)$).
\end{itemize}

Then, by Corollary~\ref{c:umf-KGamma-expansions}, $\Exp(\bX, \pkf(\bN))$ is the UMF of $G$. A point $p$ in this space is determined by choosing a root $q \in \MHP_G(W)$ (as per Theorem~\ref{th:S_GW} and Example~\ref{ex:3-transitive}) and then a linear order around each point of the form $\bigwedge \bar x$, for $\bar x$ a tuple from $X$, which has the component of the root as its least element.

In \cite{Duchesne2020}, Duchesne constructed an interesting minimal flow of $G$ consisting of linear orders on $X$, namely the \df{convex and converging} linear orders, which we proceed to describe. He only gave the construction for the infinitely-branching case but it works equally well for finitely-branching dendrites.

A linear order $<$ on $\bps$ is \df{converging} if for all triples of points $x_1, x_2, x_3$ lying on a line in this order (i.e., such that $x_2$ is between $x_1$ and $x_3$), it is not the case that $x_1 < x_3 < x_2$. If $<$ is a converging linear order and $Y \sub X$, a sequence $(x_n)_{n \in \N}$ of elements of $Y$ is called \df{$<$-minimizing in $Y$} if for all $y \in Y$, we have that for all but finitely many $n$, $x_n \leq y$. It follows from the proof of \cite{Duchesne2020}*{Lemma 7.2} that for any $x, y \in \bps$, there is a unique (possibly regular) point $m(x, y) \in [x, y]$ which is the limit of every $<$-minimizing sequence in $\bps \cap [x, y]$.

Duchesne then defines\footnote{The definition in \cite{Duchesne2020}*{Definition~7.3} contains a typo: allowing $x'$ and $y'$ to be equal to $m(x, y)$ yields a condition which is never satisfied.} a converging linear order to be \df{convex}
if whenever $x, x', y', y \in \bps$ are such that $x' \in [x, m(x, y) )$, $y' \in (m(x, y), y]$, and $x \le y$, one has $x' \le y'$. It is not immediately clear that this condition is closed, but it follows straightforwardly from the following lemma. 
\begin{lemma}
  \label{lem:convex-equiv-def}
  Let $<$ be a converging linear order on $\bps$. Then the following are equivalent:
  \begin{enumerate}
  \item \label{i:convex:Bruno} $<$ is convex;
  \item \label{i:convex:forbidden} for all $x_1, x_2, x_3, x_4 \in X$ lying on a line in this order, it is not the case that $x_2 < x_3 < x_1 < x_4$.
  \end{enumerate}
\end{lemma}
\begin{proof}
  First, notice that we can reduce to the case $x' = x$ in the definition of convex: since $<$ is converging, any $z \in [x, m(x, y) )$ satisfies $z \le x$. 

  \ref{i:convex:Bruno} $\Rightarrow$ \ref{i:convex:forbidden}. 
  Suppose that $<$ is convex and let $x_1, x_2, x_3, x_4$ be on a line, with $x_2 < x_3 < x_1 < x_4$. 
  We have three cases, depending on the position of $m(x_1, x_4)$:
  \begin{itemize}
    \item if $m(x_1, x_4) \in [x_1, x_3)$, it follows from convexity applied to $x_1, x_3, x_4$, that $x_1 < x_3$;
    \item if $m(x_1, x_4) = x_3$, then clearly $x_3 < x_2$; 
    \item if  $m(x_1, x_4) \in (x_3, x_4)$, then there exists $z$ close to $m(x_1, x_4)$ such that $B(z, x_3, x_2)$ and  $z < x_2 < x_3$, contradicting convergence.
  \end{itemize} 
  In each case we reached a contradiction.
  
  \ref{i:convex:forbidden} $\Rightarrow$ \ref{i:convex:Bruno}. Let $x, y', y$ be such that $y' \in (m(x, y), y]$ and $B(x, y', y)$.
  Suppose towards contradiction that $x < y$ and $y' < x$.
  We can find $z < y'$ close to $m(x, y)$ such that $x, z, y', y$ are on a line and $z < y' < x < y$, which is forbidden. 
\end{proof}

From now on, by a \df{convex} linear order we will mean one satisfying condition \ref{i:convex:forbidden} in Lemma~\ref{lem:convex-equiv-def}. This definition also has the advantage that it makes sense on finite trees.
We record the following useful facts about CCLOs (cf. \cite{Duchesne2020}*{Lemma 7.5}).

\begin{lemma}
  \label{lem:cclo-5-point-lemma}
  Let $<$ be a convex converging linear order on (a finite subtree of) $\bps$.
  \begin{enumerate}
    \item \label{i:lem:cclo-5-point-lemma-1}
    If $z_1, y_1, x, y_2, z_2$ are and on a line in this order (not necessarily all distinct) and such that $x < y_1 < y_2$, then $z_1 < z_2$; 
    \item \label{i:lem:cclo-5-point-lemma-2}
    Let $x < y$ and let $a_1, a_2 \in \comps{y}$ be components which do not contain $x$. For any $z_1, z_1' \in a_1$ and $z_2, z'_2 \in a_2$, if $z_1 < z_2$, then $z_1' < z_2'$. 
  \end{enumerate}
\end{lemma}
\begin{proof}
  \ref{i:lem:cclo-5-point-lemma-1}
  Since $<$ is converging, $y_1 \leq z_1$ and $y_1 < y_2 \leq z_2$. If $z_1 > z_2$, then $z_2, x, y_1, z_1$ are on a line and such that $x < y_1 < z_2 < z_1$, contradicting convexity.
  
  \ref{i:lem:cclo-5-point-lemma-2}
  Let $w_i = \centerf(y, z_i, z'_i)$ for $i = 1,2$. The hypothesis implies that $w_i \neq y$. 
  By applying the converging property twice, $x < y < w_i \le z_i$, so in particular $w_1 \le z_1 < z_2$. 
  Suppose that $w_2 < w_1$ (in particular $w_2 \ne z_2$), then $w_1, y, w_2,z_2$ are in a line and $y < w_2 < w_1 < z_2$, contradicting convexity. 
  It follows that $y < w_1 < w_2$; but $z'_1, w_1, y, w_2,z'_2$ are on a line (possibly not distinct), so we conclude by point \ref{i:lem:cclo-5-point-lemma-1}.
\end{proof}

Denote by $\CCLO(\bps)$ the space of all converging and convex linear orders on $X$. This is a closed and $G$-invariant subset of $2^{X^2}$, so it is a $G$-flow. In the infinite-branching case, Duchesne proved that this flow is minimal, that it has a comeager orbit $G \cdot \eta$, and that the stabilizer $G_{\eta}$ is extremely amenable. Then he claimed that the natural map $G/G_{\eta} \to G \cdot \eta$ is a uniform homeomorphism and concluded that $\CCLO(\bps)$ is the UMF of $G$. This last claim is, however, not correct, as the following theorem shows.

\begin{theorem}
  \label{th:CCLO}
  Let $n = 3, 4, \ldots, \infty$, let $G = \Homeo(W_n)$, and let $\bps$ denote the set of branch points of $W_n$. Then the following hold:
  \begin{enumerate}
  \item \label{i:CCLO:n-finite} If $n$ is finite, then the flow $G \actson \CCLO(X)$ is isomorphic to $\umf(G)$;
  \item \label{i:CCLO:n-infinite} If $n = \infty$, then the flow $G \actson \CCLO(X)$ is a proper factor of $\umf(G)$. In particular, $\CCLO(X)$ is not isomorphic to $\umf(G)$.
  \end{enumerate}
\end{theorem}
\begin{proof}
  We will use the description of $\umf(G)$ given by Corollary~\ref{c:umf-KGamma-expansions}, as discussed in the beginning of this section. Denote $\Xi = \Exp(\bX, \pkf(\bN))$ and a define a map $\pi \colon \Xi  \to 2^{X^2}$  by
  \begin{multline*}
    x <_{\pi(p)} y \iff p \models x \neq y \ \& \ \big( x \meet y = x  \lor {} \\
    \Phi(x \meet y, x) \prec \Phi(x \meet y, y) \big) \quad \text{for } x, y \in X.
  \end{multline*}
  Here we use the symbol $\prec$ to denote the order in the expansion $\pkf(\bN)$ in order to distinguish it from the order $\pi(p)$ on $X$ that we are defining. It is clear that this map is continuous and $G$-equivariant. 
  We check that its image is equal to $\CCLO(X)$. 
  We do so in two steps;
  first we show that there is some point $p_0 \in \Xi$ such that $\pi(p_0) \in \CCLO(X)$. 
  Then we prove that for all finite $T \sub \bps$ and any convex and converging linear order $<_T$ on $T$ there is $p \in \Xi$ such that $<_{\pi(p)}$ agrees with $<_T$ on $T$.
  By minimality of $\Xi$ and compactness of $\CCLO(X)$, this is enough.
  
  We choose $p_0$ to be the type corresponding to the expansion $\pkf(\bN)$ of $\bX$, and then the root is an endpoint in $W$. First, it is easy to check that $\pi(p_0)$ is a linear order.

  Next we check that $\pi(p_0)$ is converging. To avoid clutter, denote $<_{\pi(p_0)}$ simply by $<$. Towards a contradiction, suppose that $x, y, z \in X$ are such that $B(x, y, z)$ and $x,z < y$. We have that $x \meet z \in [x, z]$, and by symmetry, we may assume that $x \meet z \in [x, y]$. But then $y \meet z = y$, which implies that $y < z$, contradiction.

  We finally see that $<$ is convex. Towards a contradiction, suppose that $x, y, z, w$ lie on a line in this order and that $y < z < x < w$. We consider several cases depending on the position of $x \meet w$. If $x \meet w = x$, then $x \meet y = x$, so $x < y$. If $x \meet w \in (x, z)$, then $x \meet z = x \meet w$ and $\Phi(x \meet z, z) = \Phi(x \meet z, w)$, so $z$ and $w$ must be on the same side of $x$ in $<$. Finally, if $x \meet w \in [z, w]$, then $z \meet y = z$, so $z < y$. In all cases, we obtained a contradiction.

  Now, let $<_T$ be a convex and converging  linear order on a finite $T \sub \bps$, which we can suppose is center-closed and has at least two points. 
  Let $x_0$ be the minimum of $<_T$ and let $x_1$ be $<_T$-greatest of its neighbors. 
  We chose an endpoint $\xi$ as a root in a way that $x_0 \meet x_1 \in (x_0, x_1)$. 
  Call $\bS$ the $\cL^*_\bX$-structure generated by $T$ and $\xi$. 
  We now define an $\cL^*_\bX \cup \set{\prec}$ expansion $\bS'$ of $\bS$ by describing $\prec$ on each $\comps{x} \in \Comps{S}$. 
  Let $\compf(x_0 \meet x_1, x_0) \prec \compf(x_0 \meet x_1, x_1)$. 
  For $x \in T$, and any $a, a' \in \comps{x} \sminus \set{\rho(x)}$, let $\rho(x) \prec a$, and $a \prec a'$ if and only if there are $y \in a, y' \in a'$ with $y <_T y'$. 
  By \Cref{lem:cclo-5-point-lemma} \ref{i:lem:cclo-5-point-lemma-2}, $\prec$ is a linear order on each $\comps{x}$, so $\bS' \in \Age(\pkf(\bN))$.
  Let $p \in \Xi$ be any type extending the type of $\bS'$. 
  To show that $<_{\pi(p)}$ agrees with $<_T$, let $x,x' \in T$ with $x <_T x'$. 
  Since $<_T$ is converging and $x_0$ is the minimum, it cannot be that $\betr(x_0, x', x)$, so $x \meet x' \ne x'$. 
  If $x \meet x' = x$, then $x <_{\pi(p)} x'$ and we are done. 
  Suppose now that $x \meet x' = x_0 \meet x_1$. It follows from \Cref{lem:cclo-5-point-lemma} \ref{i:lem:cclo-5-point-lemma-1}, the fact that $x <_T x'$ and the fact that $x_1$ is the greatest of $x_0$'s neighbors, that $x, x_0, x_1, x'$ are on a line. So $\compf(x \meet x', x) = \compf(x_0 \meet x_1, x_0) \prec \compf(x_0 \meet x_1, x_1) = \compf(x \meet x', x')$ and thus $x <_{\pi(p)} x'$. 
  Otherwise, $x \meet x' \in T$, so by definition $\compf(x \meet x', x) \prec \compf(x \meet x', x')$ and again $x <_{\pi(p)} x'$.

  \ref{i:CCLO:n-finite} Now we suppose that $n$ is finite and show that $\pi$ is injective. Suppose that $p_1 \neq p_2$ in order to show that $\pi(p_1) \neq \pi(p_2)$. We denote by $\theta \colon \Xi \to \MHP_G(W)$ the natural factor map defined by the restriction of the type to the language of $\pkf(\bM_\cc)$ and taking $\tp(\bxi/X)$ (here we use the description of $\MHP_G(W)$ given by Theorem~\ref{th:S_GW}). First suppose that $\theta(p_1) \neq \theta(p_2)$. Then by the description of $\MHP_G(W)$ given in Example~\ref{ex:3-transitive}, there exist $x, y_1, y_2 \in X$ such that $x \in (y_1, y_2)$ and
  \begin{equation*}
    p_1 \models \Phi(x, y_1) = \rho(x) \And p_2 \models \Phi(x, y_2) = \rho(x).
  \end{equation*}
  (This is where we use that $\comps x$ is finite. In the infinite-branching case, there exist types $p$ such that $p \models \Phi(x, y) \neq \rho(x)$ for all $y \in X$; see the proof of \ref{i:CCLO:n-infinite} below.) But then
  \begin{equation*}
    p_1 \models x \meet y_2 = x \ \& \ y_1 \meet x = y_1 \meet y_2,
  \end{equation*}
  which implies that $x <_{\pi(p_1)} y_1$ and $x, y_2$ are on the same side of $y_1$ in $<_{\pi(p_1)}$. Similarly, when we replace $p_1$ by $p_2$ and exchange $y_1$ and $y_2$. These conditions imply that the orders $\pi(p_1)$ and $\pi(p_2)$ cannot agree on $x, y_1, y_2$.

  Finally, suppose that $\theta(p_1) = \theta(p_2)$. Realize $\theta(p_1)$ in some model $\bY^*$ of $\Th(\bX^*)$. In that way, we can identify $X$ with a subset of this model. We can consider $p_1$ and $p_2$ as the types of the same tuple (as usual, indexed by $X$) in two expansions of $\bY^*$ to different models of $\Th(\pkf(\bN))$. In particular, these two models have the same universe and only differ in the interpretation of $\prec$.
  Then there must exist tuples $\bar x, \bar y, \bar z$ from $X$ such that, denoting $x = \bigwedge \bar x, y = \bigwedge \bar y, z = \bigwedge \bar z$, we have that $x \in (y, z)$ and
  \begin{equation*}
    p_1 \models \Phi(x, y) \prec \Phi(x, z) \And p_2 \models \Phi(x, z) \prec \Phi(x, y).
  \end{equation*}
  (Note, however, that $x, y, z$ need not belong to $X$. They are elements of the model of $\Th(\bX^*)$, where we realized $p_1$ and $p_2$.)
  We consider the position of $y \meet z$. If $y \meet z \in [y, x)$, then $\rho(x) = \Phi(x, y)$ must be the $\prec$-least element of $\comps x$ in both $p_1$ and $p_2$, which is not possible. Similarly, it is not possible that $y \meet z \in [z, x)$, so the only possibility is that $y \meet z = x$. But then, for each $i, j$, we have that $y_i \meet z_j = x$, $\Phi(x, y_i) = \Phi(x, y), \Phi(x, z_j) = \Phi(x, z)$, so by the definition of $\pi$, $y_i <_{\pi(p_1)} z_j$ and $z_j <_{\pi(p_2)} y_i$, implying that $\pi(p_1) \neq \pi(p_2)$.

  \ref{i:CCLO:n-infinite} We consider the case $n = \infty$ and construct two types $p_1 \neq p_2 \in \Xi$ with $\pi(p_1) = \pi(p_2)$. Fix two distinct vertices $x_0, x_1 \in X$ and let
  \begin{align*}
    A &= \set{p \in \Xi : p \models \rho(x_0) = \Phi(x_0, x_1)}, \\
    B &= \set{p \in \Xi : p \models \rho(x_0) \neq \Phi(x_0, x_1)}.
  \end{align*}
  Note that $A$ and $B$ are closed and disjoint. By compactness, it suffices to construct for every finite subtree $T \sub X$ containing $x_0$ and $x_1$, two types $q_1 \in A$ and $q_2 \in B$ such that the orders $\pi(q_1)$ and $\pi(q_2)$ agree on $T$. For $q_1$, we choose an endpoint $\xi_1$ as the root in a way that $x_0 \meet x_1 \in (x_0, x_1)$ and no point in $T$ lies between $x_0 \meet x_1$ and $x_0$. 
  For $q_2$, using that the dendrite is infinitely branching, we choose a root $\xi_2$ that does not belong to any component in $\comps x_0 \cap \Comps{T}$ (so that $q_2 \models \rho(x_0) \neq \Phi(x_0, y)$ for all $y \in T \sminus \set{x_0}$).
  Let $\bS_1, \bS_2$ be the $\cL^*_{\bX}$-structures generated from $T$ and the choices of $\xi_1, \xi_2$ for $\bxi$, respectively. In particular, $S_1 = T \cup \set{\xi_1, x_0 \meet x_1}$, $S_2 = T \cup \set{\xi_2}$.
  In particular, each $x \in T \sminus \set{x_0}$ has the same set of components $\comps x$ in $\bS_1$ and in $\bS_2$.
  We can choose the same arbitrary ordering $\prec$ on $\comps x$ (for $x \in T \sminus \set{x_0}$) for both $q_1$ and $q_2$, subject to the condition that $\rho(x)$ is always the least element.
  
  For clarity, denote by $A_1, A_2$ the respective sets of components around $x_0$ in $\bS_1$ and $\bS_2$, and by $\rho_1(x_0), \rho_2(x_0)$ the components containing the root. 
  Notice that $\rho_1(x_0) = \compf(x_0, x_1)$ and that $\rho_2(x_0)$ does not appear in $A_1$. 
  For $q_1$, let $\prec_1$ be an arbitrary order on $A_1$ with minimum $\rho_1(x_0)$. 
  For $q_2$, let $\prec_2$ be the same order except that $\Phi(x_0, x_1) = \rho_1(x_0)$ is the maximum, and extend it to $A_2$ as usual, by letting $\rho_2(x_0)$ be its minimum. 
  Finally, for the order around the point $x_0 \meet x_1$ in $q_1$, we let $q_1 \models \rho(x_0 \meet x_1) \prec  \Phi(x_0 \meet x_1, x_0) \prec \Phi(x_0 \meet x_1, x_1)$.
 
  Extend $q_1$ and $q_2$ arbitrarily to full types in $\Xi$. We claim that $\pi(q_1)$ and $\pi(q_2)$ agree on $T$.
  The only non-trivial relation to check is that the orders between $x_1$ and any $y$ such that $\betr(y, x_0, x_1)$ coincide. 
  We have that $q_1 \models y \meet x_1 = x_0 \meet x_1$, and that $q_1 \models \Phi(x_0 \meet x_1, y) = \Phi(x_0 \meet x_1, x_0) \prec \Phi(x_0 \meet x_1, x_1)$, so $y <_{\pi(q_1)} x_1$. 
  On the other hand, $q_2 \models (y \meet x_1) = x_0$. 
  If $y = x_0$, by the definition of $\pi$, $y <_{\pi(q_2)} x_1$; 
  if $y \ne x_0$, then by the definition of $\prec_2$, $q_2 \models \Phi(x_0, y) \prec \Phi(x_0, x_1)$, so again $y <_{\pi(q_2)} x_1$.

  The final claim in the statement follows from the fact that every endomorphism of an UMF is an automorphism.
\end{proof}

\bibliography{main.bib}

\end{document}